\documentclass {amsart}
 
\usepackage{amsthm,amsmath,amssymb,verbatim} 
\usepackage[all]{xy}
\UseComputerModernTips 

\theoremstyle{plain} 
	\newtheorem{thm}{Theorem}[section] 
	\newtheorem{prop}[thm]{Proposition} 
	\newtheorem{lemma}[thm]{Lemma}  
	\newtheorem{cor}[thm]{Corollary} 
 
\theoremstyle{definition} 
	\newtheorem{defn}[thm]{Definition} 
	 
	\newtheorem{remark}[thm]{Remark} 
	\newtheorem{examples}[thm]{Example}

\newcommand{\catc}{\mathcal C}
\newcommand{\catd}{\mathcal D}

\newcommand{\catm}{\mathfrak M}

\newcommand{\catn}[1][n]{[#1]}

\newcommand{\sset}{\mathcal S}

\newcommand{\caln}{\mathcal N}
\newcommand{\catu}{\mathcal U}
\newcommand{\catv}{\mathcal V}
\newcommand{\catubar}{\overline{\catu}}

\DeclareMathOperator{\cat}{Cat}

\newcommand{\lofx}[2][\alpha]{\mathcal L_{#1}\left(#2 \right)}
\newcommand{\alofx}[2][\alpha]{L_{#1}\left(#2 \right)}

\newcommand{\blofx}[2][\alpha]{\delta_{#1}\left(#2 \right)}
\newcommand{\latch}[2][\alpha]{\eta_{#1}\left(#2 \right)}
\newcommand{\mofx}[2][\alpha]{\mathcal M_{#1}\left(#2 \right)}
\newcommand{\amofx}[2][\alpha]{M_{#1}\left(#2 \right)}
\newcommand{\altmofx}[2][\alpha]{M_{#1}^{\catc_0}\left(#2 \right)}
\newcommand{\bmofx}[2][\alpha]{\sigma_{#1}\left(#2 \right)}
\newcommand{\match}[2][\alpha]{\mu_{#1}\left(#2 \right)}

\newcommand{\fnc}[1][n]{F^{#1} \catc}
\newcommand{\latchcat}[1][\alpha]{\text{Latch}(\alpha)}
\newcommand{\matchcat}[1][\alpha]{\text{Match}(\alpha)}
\newcommand{\lefti}[1][i]{\mathbb L_{#1}}
\newcommand{\righti}[1][i]{\mathbb R_{#1}}

\DeclareMathOperator{\map}{map}

\DeclareMathOperator{\colim}{colim}

\newcommand{\leftmod}[1][]{\textbf{Left}(\catc_0,\catm_{#1}^\catc)}
\newcommand{\rightmod}[1][\catm]{\textbf{Right}(\catc_0,{#1}^\catc)}
\newcommand{\leftmodup}[1][]{\textbf{Left}(\catc_0,\catm^{\catc_{#1}})}
\newcommand{\leftmodc}[2][]{\textbf{Left}(\catc_{#2},\catm_{#1}^\catc)}
\newcommand{\rightmodc}[2][]{\textbf{Right}(\catc_{#2},\catm_{#1}^\catc)}

\newcommand{\diagramit}[2][]{\begin{equation}{#1}\xymatrix{#2}\end{equation}}

\begin{document} 
 
\title	[On modified Reedy/projective structures] 
	{On modified Reedy and modified projective model structures} 
 
%Author information 

\author [M. W. Johnson]
        {Mark W. Johnson}
\address{Department of Mathematics\\
         Penn State Altoona\\
         Altoona, PA 16601-3760}
\email{mwj3@psu.edu}

\begin{abstract}
Variations on the notions of Reedy model structures and projective model structures
on categories of diagrams in a model category are introduced.  
These allow one to choose only a subset of the entries
when defining weak equivalences, or to use different model categories at different entries 
of the diagrams.  As a result, a bisimplicial model category that can be used to recover 
the algebraic K-theory for any Waldhausen subcategory of a model category is produced.
\end{abstract}

%AMS information
%\keywords{} 
%\subjclass{Primary: ; Secondary }

\date{\today}

\maketitle

\section{Introduction}

Reedy model structures form the primary means of building model structures on categories 
of diagrams by imposing additional conditions on the indexing category rather than on the 
target model category.  The traditional goal is to build a model structure where weak 
equivalences of diagrams are defined by the collection of evaluation functors, such as the 
projective or standard Reedy structure.  The current article provides a method for constructing 
modified Reedy structures (Thm \ref{modReedy}), and modified projective structures 
(Prop \ref{modproj}), where weak equivalences are defined by only a subset of the evaluation 
functors.  For additional flexibility, one can also
consider different model structures at various points in the diagram in both of these results. 

After constructing and studying modified Reedy structures,
modified projective structures are introduced in order to extend the well-known Quillen 
equivalence between the standard Reedy and projective structures 
(Thm. \ref{samehtpy}).  As one might expect,
the homotopy theory of these modified structures is determined by that of the diagrams 
indexed on the full subcategory associated to the chosen subset of objects 
(Prop. \ref{restrict}).  As the
technical conditions for the existence of these model structures on diagrams are different,
one suggestion would be viewing them as different means of producing a model 
for the intended homotopy
theory.  By choosing appropriate subsets of objects, a variety of (co)localizations of the standard 
Reedy structure are produced in these left (or right) modified Reedy structures, which again is
somewhat surprising because there are no technical restrictions on the target model category. 
Among other things, two (sometimes three) model structures are given on the simplicial 
objects $s\catm$ which are each Quillen equivalent to $\catm$ itself 
(Cor. \ref{simpsame}).  Thus the localization of the standard Reedy structure on 
simplicial objects in $\catm$ considered in \cite{RSS} is Quillen equivalent to, but in
general different from, the localization of that same standard Reedy structure constructed here 
(Rem.  \ref{notexternal} and \ref{notRSS}), which is a bit surprising.
If one considers the Str{\o}m structure \cite{Strom} on topological spaces,
the model structures on the category of simplicial spaces constructed in this way seem to be 
new, with nice connections to classical homotopy theory.

Various technical properties of all of these structures are also 
considered.  Among the more obvious of these are the inheritance of various standard 
conditions, such as being proper, simplicial, or in a very special case cofibrantly generated.  
It is also shown that strong Quillen pairs, or Quillen equivalences, on the target model 
category prolong to the same in these structures.
In addition, some quite technical refinements allow a cleaner presentation of some of the 
standard arguments, most notably Prop. \ref{liftbylatch}, which applies in general
weak factorization systems.  A purely categorical observation
(Lemma \ref{simpadjt}), that for categories with zero object
the simplicial structure maps in a categorical
nerve all come in adjoint pairs $(s_i,d_{i+1})$ and $(d_i,s_i)$ for $0 \leq i \leq n-1$ along
with new pairs $(\overline s_{-1},d_0)$ and $(d_n,\overline s_{n})$, 
is another such technical improvement that seems not to be well-known.

One motivation for this work is the ability to enrich Thomason's approach
to Waldhausen's algebraic K-theory construction for a model category.  The result 
(Thm. \ref{bisimp}) is a bisimplicial object in model categories (so every structure map is
a strong left and right Quillen functor)
 such that applying an `evaluation functor' for a small full subcategory
produces the bisimplicial set for the algebraic K-theory of any 
Waldhausen subcategory (or even small subcategory of cofibrant objects).  This fits in
with the approach of \cite{DunM}, where the additional structure from an 
enrichment provided a formal approach to the
trace map, among other things.  One long term hope here 
would be to understand more of the machinery of algebraic K-theory within the broader
context of model categories, a question to be pursued in future 
joint work with Wojciech Dorabia{\l}a.  Given the large number of people using model categories
right now, it also seems likely that simplifying and generalizing two of the primary
approaches to constructing model category
structures would lead to improved technical situations in related areas, such as
higher categories, as well.

From another viewpoint, there is a reasonable amount of newly found freedom from 
the ability to consider
different model structures at different entries in a diagram, although certain relationships
between such structures are necessary at different points in this article.  
For example, given a model structure 
on $\catm$, the current techniques produce four distinct model structures on $\catm(\to)$, 
two of which are Quillen equivalent to the original $\catm$.  
If $\catm$ also has a (Bousfield) localization $\catm_f$, this leads
to an additional five distinct model structures on $\catm(\to)$, four as above starting
from $\catm_f$ and
the final one, here called a mixed structure, seems to be completely new.
This mixed structure relates well to considering just $\catm$ at the source, or $\catm_f$ 
at the target, so could be used to study localizations in a very structural way.
For example, the localization map $X \to L_f X$ is a fibrant replacement 
in this mixed structure on $\catm(\to)$ for the identity map of any 
fibrant $X \in \catm$.  This and another related example of commutative squares 
are discussed in more detail in the brief final section.  

Several other authors have recently considered extensions of Reedy's original theory, notably 
\cite{BM}, \cite{Ang}, and \cite{Bar}, in addition to two complete accounts of Reedy structures
in \cite{Phil} and \cite{GJ}.  It might be interesting to consider how to construct modified 
versions of the newer results, presumably guided by Prop. \ref{restrict}.

\subsection{Organization}

The point of section \ref{defs} is introducing the standard definitions for Reedy techniques
and one new one related to our choice of a subset of the objects.  
Appendix \ref{inductsec}
is also included to give full details on the Reedy inductive framework, since
Prop. \ref{liftbylatch} goes a bit beyond the standard claims and these details have often
been omitted in the literature.
Section \ref{buildsec} then provides the construction of left modified
Reedy structures, and the various inheritance properties 
of these are established in section \ref{inherit}.  
As expected
from the standard case, it is the `entrywise' or `internal' simplicial structure 
(even when $\catc=\Delta^{op}$) which inherits compatibility with the model structures, while
the `external' structure constructed by Quillen for simplicial objects usually does not.  
For anyone wishing to work with right modified structures, precise definitions and
statements are given (without proof) in section 5.
Section 6 is devoted to the existence and properties of modified projective structures and
various Quillen equivalence results.  Section 7 details how the current theory relates to
Waldhausen's algebraic K-theory machine and 
Thomason's variation thereof.  
Finally, section 8 provides a bit more detail about arrow categories and a related discussion
of square diagrams, and 
provides an indication of how to generalize the usual localization square of classical
homotopy theory to general model categories. 

\subsection{Acknowledgements}
I would like to thank the Math Department of Wayne State University 
for a stimulating visit to speak about 
this material.  In particular, I thank Dan Isaksen for suggesting the possibility of working 
with different model structures at various points in the diagram, after I stated 
Prop. \ref{liftbylatch}.  
Thanks are also due to John Klein for suggesting I look at Thomason's $T_\bullet$ 
construction, once I outlined my (more complicated) approach to Waldhausen's $S_\bullet$ construction.  An anonymous referee has also made a number of suggestions which 
improved the presentation.  Finally, I would like to thank my colleagues in the Penn State
Topology/Geometry Seminar for enduring my abstractions.

\section{The Reedy Structure}
\label{defs}

This section is intended to introduce mostly standard definitions, together with one new 
condition related to the choice of a full subcategory.  First is a bit of motivation for the
ideas behind Reedy indexing categories.

Suppose $i:\catd \to \catc$ is a functor between small categories, and $\catm$ is any category
with all small (co)limits.  Then $i^*$ induces a (precomposition) functor $\catm^{\catc} \to \catm^{\catd}$, which has both a left adjoint $\lefti$ and a right adjoint $\righti$ by the Kan
extension formulae.  In particular, there are units of adjunction $\lefti i^*X \to X$ and
$X \to \righti i^*X$ for each $X:\catc \to \catm$. 

The following may seem overly specialized, but instances of both types
will be created for each object in a Reedy category.
When $i$ is the inclusion of a full subcategory
missing only the final object $\delta$ in the directed $\catc$, 
the unit of adjunction $\lefti i^*X \to X$ is the identity other than
at the final object, where it is the key entry $\colim_{\catd} i^* X \to X_\delta$.  Similarly,
when $i$ is the inclusion of a full subcategory of the directed $\catc$ missing only 
the initial object, $X \to \righti i^*X$ is the identity other than at the initial object $\delta$, 
where it is the key entry $X_\delta \to \lim_{\catd} i^* X$.  Thus, in these two particular 
instances the units of adjunction instead degenerate to single maps in $\catm$.

Next is the definition of a Reedy indexing category, which is more general than a directed
category but still allows a certain form of induction as described below and in the appendix.

\begin{defn}
	\label{basicreedy}
	A Reedy category is a small category $\catc$ together with a whole number valued
	degree function on objects,
	and two subcategories, each containing all objects, $\catc^{+}$ and $\catc^{-}$ 
	such that each non-identity morphism of $\catc^{+}$ (resp. $\catc^{-}$) raises
	(resp. lowers) degree and each morphism in $\catc$ has a unique
	factorization $f=g p$ where $p \in \catc^{-}$ and $g \in \catc^{+}$.
\end{defn}

An important bit of notation is that $\fnc$ indicates the full subcategory of
$\catc$ whose objects have degree less than or equal to $n$.  
Notice this will always inherit a Reedy category structure from $\catc$ itself, since the
indicated factorizations pass through an object of degree lower than that of the source or target.

\subsection{Latching and Matching Constructions}

One can now define certain subcategories of a Reedy category, which should be thought
of as allowing attention to focus at a certain object, often acting as if it were the final object
of a directed category.

\begin{defn}
	\label{latchcat}
	Given $\alpha \in \catc$, define the \textit{latching category at $\alpha$}, or $\latchcat$, as
	the full subcategory of the restricted overcategory $\catc^{+}/\alpha$ 
	(so objects are maps $\beta \to \alpha$ in $\catc^{+}$ with commutative 
	triangles as morphisms) which does not contain $1_\alpha$.  
\end{defn}

Notice that the restricted overcategory $\catc^{+}/\alpha$ is directed and has $1_\alpha$ 
as final object, and there is an obvious inclusion $i_\alpha:\latchcat \to \catc^{+}/\alpha$.
There is an obvious functor $\catc^{+}/\alpha \to \catc$ given by
sending $\beta \to \alpha$ in $\catc^{+}$ to $\beta$, so any functor $\catc \to \catm$
can be `restricted to $\catc^{+}/\alpha$' by precomposing with this forgetful functor.
	
In particular, given $X:\catc \to \catm$, we restrict $X$ to $\catc^{+}/\alpha$ and then look at 
the key entry (in $\catm$) of the unit of adjunction associated to the 
functor $i_\alpha$ (as above). We will denote this $\alofx{X} \to X_\alpha$ and call it the 
absolute latching map, so the absolute latching space
$\alofx{X} \approx \colim_{\latchcat} i_\alpha^* X$.  

Given a morphism $f:X \to Y$ of such
diagrams in $\catm$, naturality of units of adjunction yields a commutative square
\diagramit{
\alofx{X} \ar[d]_{\alofx{f}} \ar[r] & X_\alpha \ar[d]^{f_\alpha} \\
\alofx{Y} \ar[r] & Y_\alpha  .
}
Next one extends this to a larger commutative diagram
\diagramit{
\alofx{X} \ar[d]_{\alofx{f}} \ar[r] & X_\alpha \ar[d]_{\blofx{f}} \ar[ddr]^{f_\alpha} \\
\alofx{Y} \ar[r] \ar@/_4ex/[rrd] & \lofx{f} \ar@{.>}[dr]_{\latch{f}} \\
& & Y_\alpha  
}
where $\lofx{f}$ denotes the pushout of the upper left portion, so the universal property
yields the dotted arrow and $\blofx{f}$ is part of a factorization of $f_\alpha$ as indicated.  
Here $\lofx{f}$ is the (relative) latching object, and $\latch{f}$
is the (relative) latching map (as distinct from the absolute latching object and absolute latching
map introduced for a single object above).

The primary reason for this structure is to provide an inductive framework for constructing
lifts, which are now defined.

\begin{defn}
	\label{lifting}
\begin{itemize}

\item	 Given a (solid) commutative square, 
\diagramit{
A \ar[d]_{f} \ar[r] & X \ar[d]^{p} \\
B \ar[r] \ar@{.>}[ur]^{k} & Y
}
a (dotted) diagonal $k$ making both triangles 
	commute is called a lift.
	
\item  One says $(f,p)$ have the lifting property if there exists a lift in every (solid) square 
	diagram of this form.  Note this is definitely a property of the \textit{ordered} pair.	

\end{itemize}
\end{defn}

Assume for the moment that $\catm$ is a model category.
One uses the latching constructions to define cofibrations of diagrams by
the (relative) latching maps $\latch{f}$ being cofibrations.  
A key step will be the ability to induct along
degree to show $\blofx{f}$ is then a cofibration, so $f_\alpha$ will also be a 
cofibration in $\catm$.  One can also verify lifting properties between $f$ and $p$
in terms of comparing $\latch{f}$ and a dual construction outlined below, $\match{p}$.
(See Prop. \ref{liftbylatch}.)

\begin{examples}
	If $\catc=\{0 \to 1\}$ and $\catm$ is a model category, 
	then $\catm^\catc$ is the arrow category of $\catm$ and a map
	of arrows $f:X \to Y$ may be viewed as a (distorted and decorated) commutative square.  
\diagramit{
X_0 \ar[d]_{f_0} \ar[r] & X_1 \ar[d]_{\blofx[1]{f}} \ar[ddr]^{f_1} \\
Y_0 \ar[r] \ar@/_4ex/[rrd] & \lofx[1]{f} \ar@{.>}[dr]_{\latch[1]{f}} \\
& & Y_1  
}
	Consider
	an entrywise acyclic fibration of arrows, $p:W \to Z$ with $p_0$ and $p_1$
	both acyclic fibrations in $\catm$, which fits into a (solid) lifting square
\diagramit{
X \ar[d]_{f} \ar[r] & W \ar[d]^{p} \\
Y \ar[r] \ar@{.>}[ur]^{k} & Z  .
}  
	It should be clear that a lift $k_0$ of $f_0$ against $p_0$ and
	a lift $k_1$ of $f_1$ against $p_1$ may not be sufficiently compatible to 
	define a morphism in
	the category of arrows, so a lift of $f$ against $p$.
	However a (dotted) lift $k_0$ in the diagram
\diagramit{
X_0 \ar[d]_{f_0} \ar[r] & W_0 \ar[d]^{p_0} \\
Y_0 \ar[r] \ar@{.>}[ur]^{k_0} & Z_0
}
	does induce a (solid) commutative square in $\catm$
\diagramit{
{\lofx[1]{f}} \ar[d]_{\latch[1]{f}} \ar[r] & W_1 \ar[d]^{p_1} \\
Y_1 \ar[r] \ar@{.>}[ur]^{k_1} & Z_1 
} 
	and a (dotted) lift in this second diagram includes precisely the required compatibility
	with $k_0$ in order to define a morphism
	of arrows $k:Y \to W$ which would be a lift of $f$ against $p$.  Thus, it makes more
	sense to require $\latch[1]{f}$ to have some lifting property than to consider
	only such a property for $f_1$.  
	In this example the dual notions of matching objects do not occur since
	$\catc$ is already directed (upward).
\end{examples}

Now the dual notions of matching constructions are more briefly introduced.

\begin{defn}
	\label{matchcat}
	Given $\alpha \in \catc$, define the \textit{matching category at $\alpha$}, or $\matchcat$, as
	the full subcategory of the restricted undercategory $\alpha\backslash\catc^{-}$ 
	(so objects are maps $\alpha \to \beta$ in $\catc^{-}$ while commutative triangles are
	morphisms) which does not contain (the initial object) $1_\alpha$.  
\end{defn}

Given $X:\catc \to \catm$, and $\alpha \in \catc$ 
associated to the inclusion functor $\matchcat \to \alpha\backslash\catc^{-}$
one has the key entry of the unit of adjunction
$X_\alpha \to \amofx{X}$ which will be called the absolute matching map, with target the
absolute matching object.  Given a map $f:X \to Y$ of diagrams, one has the following
commutative diagram in $\catm$
\diagramit{
X_\alpha \ar@{.>}[dr]^{\match{f}} \ar@/_2ex/[ddr]_{f_\alpha} \ar@/^4ex/[drr] \\
& \mofx{f} \ar[d]^{\bmofx{f}} \ar[r] & \amofx{X} \ar[d]^{\amofx{f}} \\
& Y_\alpha \ar[r] & \amofx{Y}
}
where $\mofx{f}$ is the pullback of the lower right portion, and the universal property induces
the dotted arrow.

Be sure to notice $f_\alpha=\bmofx{f}\match{f}$, just as $f_\alpha=\latch{f} \blofx{f}$ earlier.

Next is a convenient presentation of the inductive process, which is slightly more flexible
than the standard statements.  The added flexibility is what is needed for the current
generalizations.  While the techniques used here
are now standard, this statement goes a bit beyond previous claims and detailed proofs
have often been omitted, so a complete proof has been included as Appendix \ref{inductsec}.

\begin{prop}
	\label{liftbylatch}
	Suppose $f$ and $p$ are morphisms in $\catm^\catc$.  If $(\latch{f},\match{p})$ have the
	lifting property in $\catm$ for each $\alpha \in \catc$, then $(f,p)$ have the lifting property.
\end{prop}

Another important technical result verified by this sort of induction is the following, 
which will help when characterizing the class of acyclic (co)fibrations in the structures 
defined below.  A complete proof is provided in \cite[Lemma 15.3.9]{Phil}, 
although the statement there initially looks a bit different.

\begin{lemma}
	\label{liftbysmaller}
	Suppose $f$ is a morphism in $\catm^\catc$ and $g$ is a morphism in $\catm$.
	If $(\latch[\beta]{f},g)$ has the lifting property whenever $\beta \in \latchcat$, then
	$(\alofx{f},g)$ has the lifting property.  Dually, if $(g,\match[\beta]{f})$
	has the lifting property whenever $\beta \in \matchcat$, then
	$(g,\amofx{f})$ has the lifting property.
\end{lemma}

\subsection{Acceptable Subcategories}

Next is the new condition, related to the requirement of choosing a set of objects, or 
equivalently a full subcategory, in these constructions.  
Choosing all objects, or equivalently the whole indexing category, recovers the traditional 
Reedy structure in this context.

\begin{defn}  
	\label{acceptdef}
Suppose $\catc$ is a Reedy category and $\catm$ has all small (co)limits.
	\begin{itemize}
	
\item  The full subcategory $\catc_0 \subset \catc$ will be called \textbf{left acceptable}
	  provided it inherits a Reedy category structure such that the
	  matching objects relative to $\catc_0$ and those relative to $\catc$ 
	  are naturally isomorphic at any object $\alpha \in \catc_0$.  	 
	   
\item  The full subcategory $\catc_0 \subset \catc$ will be called \textbf{right acceptable}
	  provided it inherits a Reedy category structure such that the
	  latching objects relative to $\catc_0$ and those relative to $\catc$ 
	  are naturally isomorphic at any object $\alpha \in \catc_0$.  	 
	
\item	 The term \textbf{acceptable} will apply when $\catc_0 \subset \catc$ is both left and right
	acceptable.
	
	\end{itemize}
\end{defn}

\begin{examples}
	\label{acceptexs}
\begin{enumerate}

\item  It is clear from the definitions that $\catc_0=\fnc$ is acceptable for each $n$.

\item\label{singleton}  Given any object $\beta$ of degree zero, 
	$\catc_0=\{ \beta \}$ is acceptable.
	In fact, these are generally the only singletons which can be (left or right) acceptable as
	the latching object for this $\catc_0$ would always be an initial object, so unlikely to agree
	with the latching object with respect to all of $\catc$ unless it has degree zero.  
	Similarly, the matching objects
	with respect to $\catc_0$ would always be a final object, so unlikely to agree
	with the matching object with respect to all of $\catc$ unless it has degree zero.

\item\label{grids}  To illustrate the general principle of Lemma \ref{increase} below, 
let $\catc_{m,n}$ denote a `grid-like' directed category (or $\catn \times \catn[m]$)
\diagramit{
00 \ar[d] \ar[r] & 01 \ar[d] \ar[r] & \dots \ar[r] & 0n \ar[d] \\
10 \ar[d] \ar[r] & 11 \ar[d] \ar[r] & \dots \ar[r] & 1n \ar[d] \\
\vdots & \vdots & \vdots & \vdots \\
m0 \ar[r] & m1 \ar[r] & \dots \ar[r] & mn
}
	As one example, the degree function could be chosen to be the sum of the indices,
	with $\catc^+=\catc$ and $\catc^-$ discrete.  Taking $\catc_0$ to be 
	the first row and first column is then left acceptable by Lemma \ref{increase}.  
	In fact, this example will
	yield model categories closely related to Waldhausen's algebraic K-theory functor
	in section \ref{enrich}.
	
\end{enumerate}
\end{examples}

One can say $\catc$ is monotone increasing if there is a degree function where
no morphisms decrease degree, or equivalently the decreasing category is discrete.
The notion of monotone decreasing is dual.

\begin{lemma}
	\label{increase}
	Any full subcategory of a monotone increasing Reedy category is left acceptable.
	Dually any full subcategory of a monotone decreasing Reedy category is right acceptable.
\end{lemma}

\begin{proof}
	Whenever the decreasing category is discrete, or equivalently $\catc$ is monotone 		increasing, the matching objects are all limits of empty diagrams, hence the final object.  
	As a consequence, one has $\amofx{f}$ an isomorphism (identity of the final object)
	hence its base change $\bmofx{f}$ is an isomorphism between $\match{f}$ and
	$f_\alpha$.  Since the same is true for matching constructions relative to the
	subcategory, the matching condition for being left acceptable is then satisfied.  
	Choosing any degree function
	for the whole category, it descends to make the monotone increasing subcategory 
	$\catc_0$ a sub-Reedy category as well, and the dual case is similar. 
\end{proof}

\section{Constructing Left Modified Reedy Structures}
\label{buildsec}

Here is the construction of the left modified Reedy model category structure 
for an appropriate choice of diagram category, with detailed definitions and
statements for the dual right modified structures included in section \ref{rightmodsec}. 
The input throughout this section is a Reedy category $\catc$, together with a left acceptable
full subcategory $\catc_0$ and an $ob(\catc)$-indexed collection of model category structures 
$\catm_?$ on a fixed category 
$\catm$ which satisfy a compatibility condition as follows.  Of course, $Cof(\catm)$ indicates the
class of cofibrations in $\catm$ and similarly $Fib(\catm)$ indicates the class of fibrations.  

\begin{defn}[Left Compatibility Condition] Suppose $\alpha \in \catc$ with
	$\beta \in \latchcat$ and $\gamma \in \matchcat$.  Then
\label{compat}
\hspace{8in}
\begin{enumerate}

\item  \label{cofsincr} $Cof(\catm_\beta) \subset Cof(\catm_\alpha)$ 
\item  \label{fibsdecr} $Fib(\catm_\alpha) \subset Fib(\catm_\beta)$  
\item  \label{fibsincr} $Fib(\catm_\gamma) \subset Fib(\catm_\alpha)$
\item  \label{cofsdecr} If both $\alpha, \gamma \in \catc_0$ one has 
$Cof(\catm_\alpha) \subset Cof(\catm_\gamma)$.
\end{enumerate}
\end{defn}

Keep in mind that objects in latching or matching categories have 
smaller degrees than the indexing object, 
so for example the first portion says the class of cofibrations
is increasing in the degree as 
one moves along any chain of morphisms. 
At first glance, it appears the combination of these conditions should be that the model structure
remains constant.  However, this is far from the case, and there is a variety of
interesting examples.  

\begin{examples}
\label{leftcompex}
\begin{enumerate}

\item  Take as $\catm_?$ a fixed model structure on $\catm$ 
	(regardless of the value of ? in $ob(\catc)$).  If, in addition, one
	chooses $\catc_0=\catc$, this section will yield the original Reedy structure on 
	$\catm^\catc$ (with respect to this structure on $\catm$).  
	All other choices for $\catc_0$ will yield colocalizations of the original 
	Reedy structure when the family of model structures on $\catm$ is constant.

\item  Suppose $\catc$ is monotone increasing and
	$\catm$ is a left proper, cellular model category. Choose as $\catm_?$
	various (left Bousfield) localizations of this model structure in such a way as to
	localize more and more as the degree of the objects increases along any chain
	of maps.  Then the class of cofibrations
	remains that of the original $\catm$, the class of fibrations decreases as we localize
	so as to make more cofibrations acyclic, and the matching categories are all empty
	so the two conditions related to them are vacuously satisfied.  Hence, this family
	will satisfy the left compatibility condition.
	
\item  As a special case of (2), consider $\catc$ a commutative square
\diagramit{
00 \ar[d] \ar[r] & 01 \ar[d] \\
10 \ar[r] & 11
}
	with degree function given by the sum of the indices (monotone increasing
	so $\catc^+=\catc$ and $\catc^-$ discrete).   Then
	choose an appropriate model category $\catm_{00}$ and two localizations
	for $\catm_{01}$ and $\catm_{10}$.  Now think of forming the combined 
	localization if possible, inverting
	any cofibration inverted under either of the initial localizations, and allow this
	to be $\catm_{11}$.  Varying the choice of $\catc_0$, the model structures constructed
	here will be readily comparable to the original or any of the indicated localizations
	(see section 8 below).  

\end{enumerate}
\end{examples}

Now the definition of left modified Reedy structure can be made in terms of this input.  

\begin{defn}
	Given a Reedy category $\catc$, a left acceptable subcategory $\catc_0 \subset \catc$ 
	and an $ob(\catc)$-indexed family of model structures $\catm_?$ on $\catm$ 
	satisfying the left compatibility condition, the \textbf{left modified Reedy structure} on
	$\catm^\catc$, or $\leftmod$, consists of the classes of:
\begin{itemize}

\item	 weak equivalences, defined as those morphisms $f$ where $f_\alpha$ is a weak equivalence in 
	$\catm_\alpha$ whenever $\alpha \in \catc_0$;
\item  fibrations, defined as those morphisms $f$ where $\match{f}$ is a fibration in 
	$\catm_\alpha$ for each $\alpha$; and
\item  cofibrations, defined as those morphisms $f$ where $\latch{f}$ is a cofibration in 
	$\catm_\alpha$ for each $\alpha$, which must also be
	acyclic whenever $\alpha \notin \catc_0$.	
\end{itemize}
\end{defn}

\begin{remark}
	The fibrations in $\leftmod$ are precisely those of the standard Reedy structure
	(when $\catm_?$ is constant), while
	the cofibrations (which appear on the left in lifting diagrams) have been modified,
	hence the terminology.
\end{remark}

Notice there are now two alternative possible formulations of acyclic cofibrations, 
which must coincide if the result is to be a model category structure.  
Dually, a flexible characterization of acyclic fibrations is also necessary, which is 
actually the only point where the left acceptable condition is required.
Since the next two results are standard for the ordinary Reedy structure, a proof
is included only for the more difficult of them,
in order to make clear the dependence upon the various pieces of the compatibility 
assumption as well as the left acceptable condition.  

\begin{lemma}
	\label{acycofs}
	Suppose $\catm_?$ satisfies the left compatibility condition (parts  \eqref{cofsincr} and 
	\eqref{fibsdecr}).  Then the class of cofibrations in $\leftmod$ 
	which are also weak equivalences is characterized by $\latch{f}$
	an acyclic cofibration in $\catm_\alpha$ for each $\alpha$.  Furthermore, any 
	cofibration $f$ satisfies $f_\alpha$ a cofibration in $\catm_\alpha$ for each $\alpha$.
\end{lemma}

\begin{lemma}
	\label{acycfibs}
	Suppose $\catc_0 \subset \catc$ is left acceptable and $\catm_?$ satisfies the
	left compatibility condition (parts \eqref{fibsincr} and \eqref{cofsdecr}).  Then in $\leftmod$
	the class of fibrations which are also weak equivalences 
	is characterized by $\match{f}$ a fibration in $\catm_\alpha$ for each $\alpha$,
	which must be acyclic if $\alpha \in \catc_0$.  Furthermore, any fibration $p$ satisfies
	$p_\alpha$ a fibration in $\catm_\alpha$ for each $\alpha$.
\end{lemma}

\begin{proof}
	First,  suppose $p$ a morphism in $\catm^\catc$ with each $\match{p}$ a fibration 
	in $\catm_\alpha$.  In particular, $\match[\gamma]{f}$ is then a fibration in 
	$\catm_\alpha$ whenever $\gamma \in \matchcat$ by \eqref{fibsincr} of the left
	compatibility condition.  Then to see $p_\alpha$
	is a fibration in $\catm_\alpha$, it suffices to see $\bmofx{p}$ is a fibration
	in $\catm_\alpha$, or by closure under cobase change, that $\amofx{p}$ is
	a cofibration in $\catm_\alpha$.  However, this follows from Lemma \ref{liftbysmaller}
	by considering an arbitrary acyclic cofibration $f$ in $\catm_\alpha$, and observing
	that $(f,\match[\gamma]{p})$ has the lifting property.
				
	If $p$ has each $\match{p}$ a fibration in 
	$\catm_\alpha$ that must also
	be acyclic when $\alpha \in \catc_0$, then suppose $\gamma \in \matchcat$ with 
	$\alpha,\gamma \in \catc_0$.  By \eqref{cofsdecr} of the left compatibility condition,
	$f \in Cof(\catm_\alpha)$ implies $(f,\match[\gamma]{p})$ has the lifting property.  Due to the
	left acceptable condition, it would be equivalent to consider lifting against the matching map
	formed in the subcategory $\match[\gamma]{p}^{\catc_0}$, so by
	applying Lemma \ref{liftbysmaller} with respect to the smaller indexing category
	$\catc_0$ it follows that $(f,\altmofx{p})$ has the lifting property.  Again using the
	left acceptable condition, one concludes $(f,\amofx{p})$ also has the lifting property.  
	Since $f$ was arbitrary, it follows 
	that $\amofx{p}$, and as a consequence its base change $\bmofx{p}$,
	is an acyclic fibration in $\catm_\alpha$. 
	
	Finally, suppose $p$ is a fibration as well as a weak equivalence.  Then by definition
	each $\match{p}$ is a fibration, which must be acyclic whenever $\alpha \notin \catc_0$
	and each $p_\alpha$ with $\alpha \in \catc_0$ is a weak equivalence in $\catm_\alpha$.
	Proceed by induction on the degree of $\alpha$ to verify that $\match{p}$ is
	an acyclic fibration in $\catm_\alpha$ even when $\alpha \in \catc_0$.  
	If $\alpha \in \catc_0$ with $|\alpha|=0$, then
	$\match{p}\approx p_\alpha$ so the claim follows from the assumption that
	$p$ is a weak equivalence.
	Now assume $\match[\gamma]{p}$ is an acyclic fibration in $\catm_\gamma$ whenever 
	$\gamma \in \matchcat$, so as in the previous paragraph $\amofx{p}$, 
	hence also its base change $\bmofx{p}$, is an acyclic fibration
	in $\catm_\alpha$ (using Lemma \ref{liftbysmaller} and  \eqref{cofsdecr} of the left
	compatibility condition).  
	Then the decomposition $p_\alpha=\bmofx{p} \circ \match{p}$ and the
	2 of 3 property for weak equivalences in $\catm_\alpha$ implies
	$\match{p}$ is a weak equivalence in $\catm_\alpha$ as well.
\end{proof}

Next is the existence theorem for the left modified Reedy structure.
With all of the technical details handled already, the proof is now relatively short.

\begin{thm}
	\label{modReedy}
	Suppose $\catc_0$ is a left acceptable full subcategory of $\catc$, and $\catm_?$ is a
	family of model structures satisfying the left compatibility condition.  Then 
	$\leftmod$ is a Quillen model category.
\end{thm}

\begin{proof}
	The existence of (co)limits is well-known in this case, as they are built
	`entrywise' (e.g. \cite[Cor. to V.3.1]{MacLane}).  The 2 of 3 property and closure
	of each class under retracts is a consequence
	of the definitions, the fact that latching and matching constructions preserve retracts,
	and the same property in each $\catm_\alpha$.  
		
	Suppose $f$ is an acyclic cofibration and $p$ is a fibration.  By Lemma \ref{acycofs}
	and the definition of fibration, each $(\latch{f},\match{p})$ has the lifting property, so
	by Prop. \ref{liftbylatch} $(f,p)$ also has the lifting property.  If instead 
	$f$ is a cofibration and $p$ is an acyclic fibration, $(\latch{f},\match{p})$ still
	has the lifting property for each $\alpha$ by Lemma \ref{acycfibs} and the definition
	of cofibration, since $\latch{f}$ is acyclic if $\alpha \notin \catc_0$, whereas $\match{p}$
	is acyclic if $\alpha \in \catc_0$.  As a consequence, Prop. \ref{liftbylatch} still
	implies $(f,p)$ has the lifting property.
	
	Suppose $g:X \to Y$ is an arbitrary morphism of $\catm^\catc$, and  
	inductively produce a factorization $g=pf$ with $p$ an acyclic fibration and $f$ a 
	cofibration.  
		
	Since $\fnc[0]$ is discrete, if $\alpha \in \catc_0$ one simply chooses an appropriate 
	factorization of each $g_\alpha$, as a cofibration 
	$f_\alpha:X_\alpha \to Z_\alpha$ followed by an acyclic fibration 
	$p_\alpha:Z_\alpha \to Y_\alpha$ in $\catm_\alpha$, or if instead 
	$\alpha \notin \catc_0$ with $f_\alpha$ an acyclic cofibration and $p_\alpha$
	a fibration in $\catm_\alpha$.  
	
	Now suppose a factorization as $\fnc[n-1]$-indexed diagrams has been chosen.  Given
	$\alpha$ of degree $n$, if $\alpha \in \catc_0$, factor the induced map 
	$\lofx{f} \to \mofx{p}$ as a
	cofibration $\latch{f}$ followed by an acyclic fibration $\match{p}$ in $\catm_\alpha$,
	or if instead $\alpha \notin \catc_0$ with $\latch{f}$ an acyclic cofibration followed
	by a fibration $\match{p}$.  By
	Lemma \ref{inductfactor}, these choices suffice to define a factorization
	$X \to Z \to Y$ as $\fnc$-indexed diagrams, completing the induction step for producing
	a factorization.  Notice
	$f$ so constructed is a cofibration by definition, while Lemma \ref{acycfibs} implies 
	$p$ is an acyclic fibration.  
	
	For the other factorization, one factors in each instance as an acyclic cofibration followed
	by a fibration, and refers to Lemma \ref{acycofs} rather than \ref{acycfibs}.
\end{proof}

\section{Inheritance Properties of Left Modified Reedy Structures}
\label{inherit}

The point of this section is to indicate that many of the commonly used 
conditions in model categories are inherited under this construction.  The property of
being cofibrantly generated is surprisingly technical, so might have been avoided other
than for its potential usefulness in algebraic K-theory applications.  In order to avoid an
Eilenberg swindle forcing algebraic K-theory to vanish, one needs to impose some 
sort of finiteness condition, which can sometimes be phrased nicely using cofibrant
generation (see \cite{Sagave}).  Thus, a special case sufficient for these applications is included
near the end of this section, but cofibrant generation is not discussed for the right modified Reedy
structures at all.

The remaining conditions are relatively straightforward, so are handled first. 

\subsection{Inheriting Properness, Being Simplicial, and Quillen Pairs}

The three conditions which are inherited without undue difficulty are properness, 
compatibility with the `internal' simplicial structure, and the existence of strong Quillen pairs
or even further strong Quillen equivalences.
Properness can be split into two pieces, and either piece will be inherited, 
but only the combined statement is given here.

\begin{lemma}
	\label{proper}
	If each model category $\catm_\alpha$ is proper, then 
	$\leftmod$ is proper.
\end{lemma}

\begin{proof}
	Lemma \ref{acycofs} implies a cofibration $f$ has each $f_\alpha$ a cofibration
	in $\catm_\alpha$, and pushouts are defined entrywise in $\catm^\catc$, so left
	properness follows.  Right properness is dual. 
\end{proof}

Another property one would like to inherit would be compatibility with a simplicial structure.  
Here we use what Goerss-Jardine \cite[just above VII.2.13]{GJ} call the ``internal" structure, 
sometimes known as 
the ``entrywise" structure, which differs from the ``external" structure \cite[II.2.5 and above]{GJ} used by Quillen 
in the special case where $\catc=\Delta^{op}$.  For the sake of clarity, the following is 
to remind the reader how the relevant operations are defined, assuming a fixed 
simplicial structure $(\otimes_\catm,\hom_{\catm})$ on $\catm$ has already been chosen.
Given $X$, $Y \in \catm^{\catc}$ and $K \in \sset$, define
$X \otimes K \in \catm^{\catc}$ by $(X \otimes K)_\alpha = X_\alpha \otimes_{\catm} K$ 
and similarly
$\hom(K,Y) \in \catm^{\catc}$ by $\hom(K,Y)_\alpha = \hom_{\catm}(K,Y_\alpha)$, while
$\map(X,Y) \in \sset$ is defined by $\map(X,Y)_n= \catm^\catc(X \otimes \Delta[n],Y)$.  Then
the triple adjunction relationship is expressed by the following natural isomorphisms of simplicial
sets
\[\sset(K,\map(X,Y)) \approx \map(X \otimes K,Y) \approx \map(X,\hom(K,Y))  .
\]
  
\begin{prop}
	\label{internalsimp}
	If $\catm$ has a simplicial structure in which each
	$\catm_\alpha$ is a simplicial model category, 
	then the ``internal" simplicial structure described above
	makes any left modified Reedy structure (which exists) into a simplicial model category.
\end{prop}

\begin{proof}
	Suppose $f:X \to Y$ is a left modified Reedy cofibration in $\catm^{\catc}$ 
	and $j:K \to L$ is a
	cofibration in $\sset$.  One must show the induced map in $\catm^{\catc}$
\[ X \otimes L \coprod_{X \otimes K} Y \otimes K \to Y \otimes L
\]
	is a left modified Reedy cofibration, which is acyclic provided either $f$ or $j$ is acyclic.
	Evaluating at an object $\alpha$ of degree zero, notice every construction is entrywise,
	so one simply has the pushout-product in $\catm_\alpha$ of $f_\alpha$ and $j$.  Thus,
	$\catm_\alpha$ simplicial implies the result is a cofibration which is acylic provided
	either $f_\alpha$ or $j$ is acyclic.
	
	Now, consider an $\alpha$ of non-zero degree.  Then
	the (relative) latching map under consideration
\[ \left[ (X \otimes L \coprod_{X \otimes K} Y \otimes K)_\alpha 
   \coprod_{L_\alpha(X \otimes L \coprod_{X \otimes K} Y \otimes K)} L_\alpha(Y \otimes L)
 \right] \to (Y \otimes L)_\alpha
\]	
	 is isomorphic to the pushout-product in $\catm$ of $\latch{f}$ and the map $j$
\[\left[ (Y_\alpha \otimes K) \coprod_{(L_\alpha Y \coprod_{L_\alpha X} X_\alpha) \otimes K} 
((L_\alpha Y \coprod_{L_\alpha X} X_\alpha) \otimes L)\right] \to Y_\alpha \otimes L	
\]
	 by compatibility of colimits (see \cite[end of 15.3.16]{Phil}).  
	 Since each $\catm_\alpha$ is a simplicial model category, the result
	 is thus a cofibration in $\catm_\alpha$ 
	 which is acyclic whenever either $j$ or $\latch{f}$ is acyclic.
	 Now the claim follows from the definition of cofibration and 
	 Lemma \ref{acycofs}.
\end{proof}

\begin{remark}
\label{notexternal}
	In modified Reedy structures, the external simplicial structure
	will rarely be fully compatible.  The standard reason given is that
	tensoring with the 	
	acyclic cofibration of simplicial sets given by
	`the lowest' $d^0:\Delta[0] \to \Delta[1]$ would require that for every 
	cofibrant $Z \in \catm^\catc$ and $n \in \catc_0$
\[ Z_n \approx (Z \otimes \Delta[0])_n \to 
	(Z \otimes \Delta[1])_n \approx \coprod_{\Delta[1]_n} Z_n
\]
	is a weak equivalence in $\catm_n$, which should rarely hold.	
	This property is important, as it distinguishes $\rightmod$ with
	$\catc=\Delta^{op}$ and $\catc_0=[0]$ from another localization of the standard
	Reedy structure considered by \cite{RSS}, even though they are (indirectly)
	Quillen equivalent (see Rem. \ref{notRSS}).
\end{remark}

The next inheritance question considered involves prolonging strong Quillen pairs and 
Quillen equivalences from the target model category.  
The interesting point here is that one only needs
the Quillen equivalence condition at entries of the subcategory $\catc_0$ in order to deduce
a prolonged Quillen equivalence for left modified structures.

 \begin{prop}
	\label{qpair}
	Suppose $(F,G)$ forms an adjoint pair between $\catm$ and $\caln$, 
	such that they become a strong Quillen pair between the model categories
	$\catm_\alpha$ and $\caln_\alpha$ for each $\alpha \in \catc$.	
	Then their prolongations $(F_*,G_*)$
	induce a strong Quillen pair between 
	$\leftmod$ and $\leftmod[\caln]$ (if both exist).
	If, in addition, $(F,G)$ is a strong Quillen equivalence for each $\alpha \in \catc_0$, 
	then $(F_*,G_*)$ is a strong Quillen equivalence as well.
\end{prop}

\begin{proof}
	The left adjoint $F$ must preserve colimits, cofibrations, and acyclic cofibrations as a strong
	left Quillen functor.  In this case, the prolongation $F_*$ will commute with 
	latching constructions and so preserve cofibrations
	by definition and acyclic cofibrations by Lemma \ref{acycofs}. 
	
	For the Quillen equivalence condition, suppose $X$ is cofibrant in 
	$\leftmod$ and
	$Y$ is fibrant in $\leftmod[\caln]$.  Then for each $\alpha \in \catc_0$, 
	$X_\alpha$ is cofibrant in $\catm_\alpha$ by Lemma \ref{acycofs} and 
	$Y_\alpha$ is fibrant in $\caln_\alpha$
	by Lemma \ref{acycfibs}.  Now 
	the claim follows from the definition of weak equivalence and the assumption of $(F,G)$
	a Quillen equivalence for each $\alpha \in \catc_0$.
\end{proof}

\subsection{A Special Case of Inheriting Cofibrant Generation}

The final property whose inheritance is considered is being cofibrantly generated, 
which is clearly \textit{not} a self-dual
condition by its nature.  In fact, it becomes quite technical to pursue this condition
in general, which will be avoided here, so only the minimum necessary for potential
applications to algebraic K-theory will be handled in this subsection.
As a consequence, 
the focus will be on the case of a single cofibrantly generated model category structure on 
$\catm$, with $\catc$ monotone increasing, throughout this subsection.  
As is customary, $I$ will denote
a set of generating cofibrations, and $J$ a set of generating acyclic cofibrations in $\catm$,
but it will also be convenient to make the (often satisfied but) 
non-standard additional assumption that $J \subset I$.  
This is really just
a way of hiding the technical assumption that the domains of $J$ are small with respect to
the subcategory of $I$-cofibrations, as noted near the end of the proof of 
Prop. \ref{increasecofgen}.

Notice each evaluation functor $ev_{\alpha}:\catm^\catc \to \catm$ has a left adjoint, defined by
\[({\bf F}^{\alpha} X)_{\beta}=
\begin{cases}
\sqcup_{\alpha \to \beta} X, &\text{if $\catc(\alpha,\beta) \neq \emptyset$;} \\
\emptyset, &\text{otherwise.}
\end{cases}
\]
with the relevant initial map or summand identity for structure maps.  
Similarly, there are also right adjoints
to evaluations defined as either products of the given object ``before" the chosen entry, 
or the final object otherwise.  Both are instances of the usual Kan extension formula 
\cite[X.3]{MacLane}.

\begin{prop}
	\label{increasecofgen}
	Suppose $\catc$ is monotone increasing and $\catm$ is a 
	cofibrantly generated model category,
	with $\catc_0$ any full subcategory of $\catc$ and $J \subset I$.  
	Then $\leftmod$	
	is a cofibrantly generated model category with 
\[ I_L=\cup_{\alpha \in \catc_0} {\bf F}^\alpha(I) \bigcup 
		\cup_{\alpha \notin \catc_0} {\bf F}^\alpha(J) 
		\text{ and } J_L=\cup_{\alpha \in \catc} {\bf F}^\alpha(J) 
\]	
	as set of generating (acyclic) cofibrations.
\end{prop}

\begin{proof}
	First, notice this model structure exists by Thm \ref{modReedy}, since 
	Lemma \ref{increase} says $\catc_0$ is left acceptable.  The
	monotone increasing assumption also implies 
	$ev_\alpha \approx \match{?}$ by $\amofx{?}$ constant on the final object.
	Thus, if these sets
	admit the small object argument, their role as generating (acylic) cofibrations is
	essentially a restatement of Lemma \ref{acycfibs} and the definition of fibrations
	via the $({\bf F}^\alpha,ev_\alpha)$ adjunction.
	
	To see these sets permit the small object argument, notice that ${\bf F}^\beta C$
	is small with respect to a set of maps $S$ in $\catm^\catc$ 
	provided $C$ is small in $\catm$ with respect
	to $ev_\beta S$.  As a consequence, $J_L$ will permit the 
	small object argument provided
	$ev_\beta(J_L)$ consists of acyclic cofibrations in $\catm$, 
	since the domains of the maps in $J$ are small with respect to the whole class of  
	acyclic cofibrations in $\catm$ by \cite[10.5.27]{Phil} (rather than just with respect to the
	relative cell complexes built using $J$).
	However, $ev_\beta {\bf F}^\alpha j$ is either the identity of
	the initial object, or else a coproduct of copies of $j$. 
	Hence, $ev_\beta {\bf F}^\alpha j$ is an acyclic cofibration in $\catm$
	whenever $j \in J$, which suffices.  
		
	The argument for
	$I_L$ allowing the small object argument is similar, but complicated by the fact
	that domains of $J$ need not be small with respect to $I$-cofibrations in general.
	However, this follows from the stronger assumption that $J \subset I$ (and $I$ allows
	the small object argument).
\end{proof}

\begin{remark}
\label{nicecofgen}
In the special case of $\catc=\Delta^{op}$, one of the two initial cases of interest in 
Bousfield-Kan \cite{BK} and then Reedy \cite{Reedy}, a different left adjoint to 
matching objects also allows one to give an explicit
set of generating (acyclic) cofibrations.
\end{remark}

\section{Statements for Right Modified Reedy Structures}
\label{rightmodsec}

Rather than trying to state each definition and result to this point (with the exception 
of Prop. \ref{increasecofgen}) in two parts, this section serves to
include clear statements for anyone working with right modified structures.

\begin{defn}[Right Compatibility Condition] Suppose $\alpha \in \catc$ with
	$\beta \in \latchcat$ and $\gamma \in \matchcat$.  Then
\begin{enumerate}
\item    $Fib(\catm_\gamma) \subset Fib(\catm_\alpha)$
\item   $Cof(\catm_\alpha) \subset Cof(\catm_\gamma)$  
\item   $Cof(\catm_\beta) \subset Cof(\catm_\alpha)$
\item   If both $\alpha, \beta \in \catc_0$ one has 
$Fib(\catm_\alpha) \subset Fib(\catm_\beta)$.
\end{enumerate}
\end{defn}

The analog of example \ref{leftcompex}(2) in this context is as follows.
	Suppose one chooses a monotone decreasing $\catc$ and successively
	colocalize, or take right Bousfield localizations, as the degree increases. 
	Then the class of fibrations remains fixed and the class of cofibrations 
	gradually shrinks as more fibrations are made acyclic, 
	while the latching categories are all empty.  Thus, the Right Compatibility
	Condition would be satisfied in this case.	 
	
	A commutative square remains a special case of potential interest, comparing two 
	colocalizations, but the common target would have degree zero,
	the intermediate objects have degree one, and the common source have degree two.
	Thus, $\catc^-=\catc$ and $\catc^+$ is discrete, with any choice of $\catc_0$
	right acceptable by Lemma \ref{increase}.

\begin{defn}
	Given a Reedy category $\catc$, a right acceptable subcategory $\catc_0 \subset \catc$ 
	and an $ob(\catc)$-indexed family of model structures $\catm_?$ on $\catm$ 
	satisfying the right compatibility condition, the \textbf{right modified Reedy structure} on
	$\catm^\catc$, or $\rightmod$, consists of the classes of:
\begin{itemize}

\item	 weak equivalences, defined as those morphisms $f$ where $f_\alpha$ is a weak equivalence in 
	$\catm_\alpha$ whenever $\alpha \in \catc_0$;
\item  cofibrations, defined as those morphisms $f$ where $\latch{f}$ is a cofibration in 
	$\catm_\alpha$ for each $\alpha$; and
\item  fibrations, defined as those morphisms $f$ where $\match{f}$ is a fibration in 
	$\catm_\alpha$ for each $\alpha$, which must also be
	acyclic whenever $\alpha \notin \catc_0$.	
\end{itemize}
\end{defn}

\begin{lemma}
	Suppose $\catm_?$ satisfies the right compatibility condition 
	(parts  \eqref{cofsincr} and \eqref{fibsdecr}).
	Then the class of fibrations in $\rightmod$ 
	which are also weak equivalences is characterized by $\match{p}$
	an acyclic fibration in $\catm_\alpha$ for each $\alpha$.  Furthermore, any 
	fibration $p$ satisfies $p_\alpha$ a fibration in $\catm_\alpha$ for each $\alpha$.
\end{lemma}

In the right modified case, identifying the acyclic 
cofibrations is where the right acceptable condition is necessary.  

\begin{lemma}
	Suppose $\catc_0 \subset \catc$ is right acceptable and $\catm_?$ satisfies the
	right compatibility condition (parts \eqref{fibsincr} and \eqref{cofsdecr}).  Then in $\rightmod$
	the class of cofibrations which are also weak equivalences 
	is characterized by $\latch{f}$ a cofibration in $\catm_\alpha$ for each $\alpha$,
	which must be acyclic if $\alpha \in \catc_0$.  Furthermore, any cofibration $f$ satisfies
	$f_\alpha$ a cofibration in $\catm_\alpha$ for each $\alpha$.
\end{lemma}

\begin{thm}
	Suppose $\catc_0$ is a right acceptable full subcategory of $\catc$, and $\catm_?$ is a
	family of model structures satisfying the right compatibility condition.  Then 
	$\rightmod$ is a Quillen model category.
\end{thm}

\begin{lemma}
	If each model category $\catm_\alpha$ is proper, then 
	$\rightmod$ is proper.
\end{lemma}

\begin{prop}
	If $\catm$ has a simplicial structure in which each
	$\catm_\alpha$ is a simplicial model category, 
	then the ``internal" simplicial structure described above
	makes any right modified Reedy structure (which exists) into a simplicial model category.
\end{prop}

Once again, notice the Quillen equivalence assumption for just the entries in the subcategory
suffices to produce a Quillen equivalence between right modified structures.

 \begin{prop}
	Suppose $(F,G)$ forms an adjoint pair between $\catm$ and $\caln$, 
	such that they become a strong Quillen pair between the model categories
	$\catm_\alpha$ and $\caln_\alpha$ for each $\alpha \in \catc$.	
	Then their prolongations $(F_*,G_*)$
	induce a strong Quillen pair between 
	$\rightmod$ and $\rightmod[\caln]$ (if both exist).
	If, in addition, $(F,G)$ is a strong Quillen equivalence for each $\alpha \in \catc_0$, 
	then $(F_*,G_*)$ is a strong Quillen equivalence as well.
\end{prop}

\section{Modified Projective Structures}

Another familiar fact is that the standard Reedy structure is Quillen equivalent to 
the projective (or diagram) model structure on $\catm^{\catc}$ when both exist.  
In order to generalize this fact, one first needs to introduce modified projective structures, 
after a small technical digression.  

\begin{remark}
	It is easy to show the intersection of the three distinguished classes in a model 
	category are precisely the isomorphisms, characterized as those $f$ where $(f,f)$ has
	the lifting property.  This is useful to keep in mind when working with various `trivial'
	model category structures.  For example, together with the lifting properties, it implies
	weak equivalences in a model category are precisely the isomorphisms if and only if 
	all maps are both cofibrations and fibrations, hence a rigidity result for the most 
	commonly used trivial model structure.
\end{remark}

\begin{defn}
	\label{lopsidedtriv}
	Let $\catm_\emptyset$ denote the (co)complete 
	category $\catm$ equipped with the following rather 
	trivial model category structure.  All maps are both fibrations and weak equivalences,
	while the cofibrations are simply the isomorphisms.  In fact, this is cofibrantly generated
	with the empty set of generating (acyclic) cofibrations, hence the notation. 
\end{defn}

Notice there is also a dual trivial model structure on any (co)complete category, with all maps acyclic cofibrations and with fibrations characterized as the isomorphisms, but it would be naturally fibrantly, rather than cofibrantly, generated.

Next is the existence theorem for modified projective structures.  Such structures can be
used in various places to provide flexibility in comparing model structures.  
As one example, in \cite{JY} it is shown that (with a fixed model structure on the 
target category)
a modified projective structure on colored 
PROPs is Quillen equivalent to the usual projective structure on colored operads, 
hence the full projective model structure on PROPs is in some sense a refinement of the 
projective structure on operads.  The current definition is more general than is common, 
in allowing various different model structures rather than a fixed one for $\catm$.   
The point is to be able to generalize the usual close
relationship between Reedy and projective model structures, as well as allowing much more
flexibility in studying the homotopy theory of diagrams. 

\begin{defn}
	\label{othercompat}
	Say a collection of model structures $\catm_\alpha$ on a fixed $\catm$ has fibrations
	which decrease along the indexing subcategory $\catc_0$ if $\catc_0(\beta,\alpha)$ 
	non-empty implies $Fib(\catm_\alpha) \subset Fib(\catm_\beta)$.
\end{defn}

Notice in this case the acyclic cofibrations increase along the subcategory, or
$AcycCofs(\catm_\beta) \subset AcycCofs(\catm_\alpha)$ if $\catc_0(\beta,\alpha)$ non-empty, 
by the lifting characterization with respect to fibrations.  
If $\catc$ is monotone increasing, this condition is implied by the left compatibility 
condition, since $\catc_0(\beta,\alpha)$ non-empty and $\beta \neq \alpha$ must then 
imply $\beta \in \latchcat$.

The proof below essentially comes from \cite[11.6.1]{Phil}, 
but is included mainly for convenience and
to clarify notation.  Here $I_\alpha$ will indicate the set of generating cofibrations for
$\catm_\alpha$, and similarly with $J_\alpha$ the generating acyclic cofibrations, which 
conflicts with the notation of \cite{Phil}.
 
\newcommand{\projmod}{\textbf{Proj}(\catc_0,\catm^\catc)}

\begin{prop}
	\label{modproj}
	Suppose each $\catm_\alpha$ is a cofibrantly generated model category, and the
	collection has fibrations which decrease along the indexing subcategory $\catc_0$.
	Then there is a modified projective 
	model structure $\projmod$	
	on $\catm^{\catc}$, with fibrations (resp. weak equivalences) defined as those maps sent to
	fibrations (resp. weak equivalences) by each $ev_\alpha$ with $\alpha \in \catc_0$.
	Furthermore, the sets of generating (acyclic) cofibrations are  
\[ I_{\catc_0}=\cup_{\alpha \in \catc_0} {\bf F}^\alpha(I_\alpha) 
		\text{ and } J_{\catc_0}=\cup_{\alpha \in \catc_0} {\bf F}^\alpha(J_\alpha)  . 
\]		
\end{prop}

\begin{proof}
	First, notice \cite[Props. 7.1.7 and 11.1.10]{Phil} 
\[ \prod_{\alpha \in \text{Ob}(\catc_0)} \catm_\alpha \times 
\prod_{\alpha \notin \text{Ob}(\catc_0)} \catm_\emptyset
\] 	is itself a cofibrantly generated model category, with generating sets
\[I_1=\cup_{\alpha \in \catc}(I_\alpha \times \prod_{\beta \neq \alpha}1_\beta)
	\text{ and } 
	J_1=\cup_{\alpha \in \catc}(J_\alpha \times \prod_{\beta \neq \alpha}1_\beta)
\]
	where $1_\beta$ is the identity of the initial object of $\catm_\beta$.
	If ${\bf F}$ is the left adjoint to the
	forgetful functor ${\bf U}:\catm^\catc \to \catm^{\catc^{\text{disc}}}$, then	
	the image of the generating cofibrations 
	${\bf F}(I_1)=I_{\catc_0}$ since $I_\alpha$ is empty for $\alpha \notin \catc_0$ by 		
	construction and similarly $J_{\catc_0}={\bf F}(J_1)$.	
	Thus, it will suffice to show one can lift this model structure
	from $ \catm^{\catc^{\text{disc}}}$ to $\catm^\catc$	
	over the adjoint pair $({\bf F}, {\bf U})$ to complete the proof.  
	
	Now notice these sets allow the small object argument, 
	just as for $J_L$ in the proof of \ref{increasecofgen}.  Also, if $\beta \in \catc_0$,
	$ev_\beta (j)$ for $j \in J_{\catc_0}$ is an acyclic cofibration in $\catm_\beta$. 
	This follows since 
	$ev_\beta {\bf F}^\alpha (j_\alpha)=\coprod_{\catc_0(\alpha,\beta)} j_\alpha$ or 
	$1_\beta$.  By construction, this is an acyclic cofibration in $\catm_\alpha$,
	so by the assumption of fibrations decreasing along the subcategory $\catc_0$, 
	an acyclic cofibration in $\catm_\beta$.
	As a consequence, ${\bf U}$ takes relative $J_{\catc_0}$-cell complexes to
	weak equivalences, and one can apply \cite[11.3.2]{Phil}.
\end{proof}

\begin{remark}
\label{sameincrease}
If $\catc$ is monotone increasing, the relative matching maps are isomorphic to the entries
by triviality of the absolute matching objects (as limits over empty categories).  Hence,
one has $\leftmod$ isomorphic (not just equivalent) to $\projmod$, since they have precisely the 
same fibrations and weak equivalences.  This is well-known for the standard Reedy and
projective structures, in this language the case $\catc_0=\catc$.
\end{remark}

Now one has the anticipated comparison result.  

\begin{thm}
	\label{samehtpy}
	Suppose $\catm_\alpha$ is a collection of model category structures
	and $\catc$ is a Reedy category equipped with a choice of full subcategory $\catc_0$.
\begin{itemize}
\item  If both structures exist, then the identity $1:\rightmod \to \leftmod$ is the right half of a
	strong Quillen equivalence.
\item  If the structures exist and $\beta \in \matchcat$ implies 
	$Fib(\catm_\beta) \subset Fib(\catm_\alpha)$, 
	then the identity $1:\leftmod \to \projmod$ (resp. 
	$1:\rightmod \to \projmod$) is the right half of a strong Quillen equivalence.
\end{itemize}
\end{thm}

\begin{proof}
	In each case, the model structures being compared have the same class of weak
	equivalences, so it suffices to show the identity preserves fibrations when considered
	as a functor in the appropriate direction.  For the first claim, this follows from the definitions
	and for the second claim this follows from the entrywise fibration portion of 
	Lemma \ref{acycfibs},
	which requires only this one part of the compatibility assumption.
\end{proof}

\begin{remark}
Keeping in mind that the three structures considered
in Thm. \ref{samehtpy} exist under different technical assumptions, 
this result might be viewed as providing alternative
existence criteria for a convenient model of the common homotopy category.   The large
number of different model structures which may be constructed by these methods should
make this additional flexibility quite useful.

Notice $\projmod$ obviously inherits the right proper condition, since pullbacks, fibrations,
and weak equivalences are defined in terms of (certain) entries.  
Once cofibrations are shown to be preserved by evaluations in $\catc_0$ as before by
considering the generating cofibrations, $\projmod$ also inherits the left proper condition since pushouts are also defined entrywise.  For the ``internal" simplicial structure
each entry of the pullback-product construction for diagrams
is isomorphic to the pullback-product construction for that entry (see \cite[11.7.3]{Phil}) and so 
$\projmod$ will also be simplicial when each $\catm_\alpha$ is a simplicial model category.
As for modified Reedy structures, the ``external" simplicial structure will rarely be fully 
compatible with modified projective structures.
\end{remark}

Next is the rather appealing fact that $\catm^{\catc_0}$ really determines the
homotopy theory of $\leftmod$. Similar results hold for $\rightmod$ and $\projmod$ 
as well, although forgetful functors are normally not strong left Quillen functors for 
modified projective 
structures.  In some sense, the proposition says $\leftmod$ is
essentially just lifting the standard Reedy structure from $\catm^{\catc_0}$, without any
of the technical conditions on the target model category normally associated 
with lifting techniques.

\begin{prop}
	\label{restrict}
	If $\catc_0$ is left acceptable, then the forgetful functor 
	$U:\leftmod$ to $\leftmodup[0]$ (the standard Reedy structure on the 
	smaller diagram category) is the right half of a strong Quillen
	equivalence.  If $\catc_0$ is acceptable, then $U$ is also the left half of a
	strong Quillen equivalence
\end{prop} 

\begin{proof}
	Recall the left Kan extension formula gives a left adjoint $L$ to the forgetful functor
	$U:\catm^\catc \to \catm^{\catc_0}$.  The left acceptable condition says that the forgetful
	functor preserves fibrations in this context, and it preserves (in fact it reflects)
	weak equivalences (between fibrant objects) by definition.  Thus, it is a strong Quillen
	pair which, by \cite[Cor.1.3.16]{Hovey}, is a Quillen equivalence provided the
	derived unit $X \to UQLX$ is a weak equivalence for each cofibrant object 
	$X \in \leftmodup[0]$, where $Q$ indicates a fibrant replacement.  Since $U$
	preserves all weak equivalences, it is enough to instead consider the unit of
	adjunction $X \to ULX$, which is an isomorphism by $\catc_0$ a full subcategory
	(see \cite[Cor. X.3.3]{MacLane}).
	
	If, in addition, $\catc_0$ is right acceptable, then the forgetful functor also preserves
	cofibrations by construction, so the dual argument applies.
\end{proof}

Next is an observation about a special case, which allows one to recover the homotopy 
theory of the original category within the context of a diagram category in many instances.  
This should be particularly useful in combination with choosing appropriate 
(co)localizations for the different 
$\catm_\alpha$, or for simplicial objects over a model category which is not cofibrantly generated,
such as the Str{\o}m structure \cite{Strom} on topological spaces.

\begin{remark}
	\label{zero}
	Suppose $\catc_0$ consists of a singleton $\beta$ which taken alone is 
	left (resp. right) acceptable
	and has no non-trivial endomorphisms, e.g. where $|\beta|=0$ as in  
	\ref{acceptexs}\eqref{singleton}.  
	Then $({\bf F}^\beta,ev_\beta)$ yields a strong Quillen equivalence
	between $\catm_\beta$ and $\leftmod$ (resp. $\rightmod$ or
	$\projmod$).
\end{remark}	

One application is related to the construction of \cite{RSS}, the current result being
somewhat more general, but much weaker by missing the key property for their application.  
Keep in mind that $\fnc[0]$ is always an acceptable subcategory.

\begin{cor}
	\label{simpsame}
	For any model category $\catm$, there are two different model structures, 
	$\leftmod$ and $\rightmod$ with $\catc_0=[0]$, on
	the simplicial objects $\catm^{\Delta^{op}}$ for which $(const,ev_0)$ and
	$(ev_0,R_0)$ each form
	a strong Quillen equivalence with $\catm$.  If, in addition, $\catm$ is cofibrantly
	generated, then a third is given by $\projmod$.
\end{cor}

It may be helpful to recall the right adjoint $R_0$ to $ev_0$ can be written explicitly as a 
power object related to the cosimplicial set $\Delta([0],?)$.  
Here each degeneracy is built from a product of diagonals and identities, 
while the face map $d_i$ comes from projection to those factors whose indices lie in 
the image of $d^i$.

\begin{remark}
	\label{notRSS}
	While the structures considered here are compatible with the internal simplicial
	structure whenever $\catm$
	itself is simplicial, the different localization of the Reedy structure considered by \cite{RSS}
	is always simplicial in Quillen's `external' structure.  Hence, their
	structure provides a simplicial model category Quillen equivalent to $\catm$ 
	whenever it exists, even if $\catm$ itself were not simplicial.   Thus 
	$\rightmod$ and the localization of the Reedy structure considered by \cite{RSS}
	may provide the first
	interesting example of two localizations of the same model structure (the standard
	Reedy structure on simplicial objects) which are (indirectly) strongly Quillen equivalent 
	but rarely coincide (as $\rightmod$ is not simplicial in the `external' structure
	in most cases as discussed in Rem. \ref{notexternal}).
\end{remark}

\section{An Enrichment of Algebraic K-theory for a pointed model category}
\label{enrich}

One complication for this section is the problem of deciding how to define `finiteness' so
as to avoid the Eilenberg swindle, or to keep the notion of algebraic K-theory non-trivial.
Working with the cofibrant, homotopy finite objects, following \cite{Sagave}, requires some
hypotheses to be sure one is working with a Waldhausen category.  The advantage is that
the structure is uniquely associated to the model category.  Working with a choice of complete
Waldhausen subcategory, following \cite{DS1}, is quite flexible, but a canonical choice is
only given for stable model categories.  Here the latter will be pursued, although the reader
is warned that the resulting algebraic K-theory space could, at this point, 
depend on more choices than just the underlying pointed model category.

\subsection{Dugger-Shipley Approach to Finiteness}

First, a brief review of the relevant details from \cite{DS1}.  
Given a subcategory $\catu$ of
$\catm$,  let $\catubar$ denote the full subcategory of $\catm$ consisting of cofibrant objects
weakly equivalent to objects of $\catu$.  If all objects of $\catu$ are cofibrant, call $\catu$
\textit{complete} if $\catu=\catubar$.  A \textit{Waldhausen subcategory} of $\catm$ will
denote a pointed (i.e. including the zero object)
full subcategory $\catu$ of cofibrant objects which is closed under homotopy 
pushouts, in the sense that the pushout $P$ formed in $\catm$ of
\diagramit{
A \ar[d]_{f} \ar[r]^g & B \ar[d] \\
C \ar[r] & P 
}
lies in $\catu$ provided $A, B, C \in \catu$ and at least one of $f$ or $g$ is a cofibration.

It is then shown in \cite{DS1} that for a strong Quillen equivalence 
$L:\catm \leftrightarrow \caln:R$ and a complete Waldhausen subcategory $\catu$ of $\catm$,
$\overline{L\ \catu}$ is a complete Waldhausen subcategory of $\caln$.  In fact, they go
on to show the induced map is an isomorphism on algebraic K-theory, and that result could
be recovered as an application of the theory which follows.

The following slight extension of \cite[Prop. 3.6]{DS1} 
is implicit in \cite[Rem. 3.7 \& Lemma A.1]{DS1}.

\begin{prop}
	\label{DSimage}
	Given Quillen equivalent model categories $\catm$ and $\caln$, together with a
	complete Waldhausen subcategory $\catu \subset \catm$, there is an `image'
	complete Waldhausen subcategory $\catv \subset \caln$ and weakly equivalent
	algebraic K-theory spaces $K(\catu)$ and $K(\catv)$. 
\end{prop}	

Once again, if the model category $\catm$ is stable, there is a canonical choice of complete
Waldhausen subcategory, coming from the compact objects (defined with respect to the 
triangulated homotopy category of $\catm$).  If both $\catm$ and $\caln$ are stable, with $\catu$
the compact objects of $\catm$ and $L:\catm \to \caln$, then $\overline{L\ \catu}$ agrees 
with the compact objects
of $\caln$ by \cite[Cor. 3.9]{DS1}, so any syzygy of Quillen equivalences will be compatible
with the canonical choice of compact objects.  However, without the assumption of stability,
the situation remains less clear.

\subsection{The $T_\bullet$ bisimplicial model category}

Throughout this subsection, the categories under consideration will be $\leftmodup[n,m]$
with $\catc_{n,m}$ as in Ex. \ref{acceptexs}(3) and
$\catc_0$ chosen to be the combination of the first row and first column.  The idea is
to use Thomason's variant of Waldhausen's construction to build a bisimplicial model category
from which one recovers the algebraic K-theory of $\catu$ by choosing the set (at least after
an appropriately large choice of universe) of cofibrant objects in $\leftmodup[n,m]$
with entries in $\catu$.  This
choice of universe business
is reasonably convenient now, since no significant cardinality arguments were
involved in developing $\leftmodup[n,m]$.  One may object that the assumption that model
categories contain all small (co)limits forbids us from changing the universe here, but within
this one section it is instead convenient to expand to Quillen's original assumption,
that only finite (co)limits are necessary.

When working with Waldhausen's $S_\bullet$ construction, it is technically important that the 
simplicial face and degeneracy maps are all exact functors.  Recall that a left adjoint is an
exact functor precisely when it preserves the distinguished classes of cofibrations and of
weak equivalences.  Thus, the natural generalization
of this condition to a bisimplicial model category would be that each face and degeneracy
map is a strong left Quillen functor.  In fact, even more is true here, as each face and degeneracy
map is also a strong right Quillen functor.  As a consequence, the entire bisimplicial 
structure descends to the level of homotopy categories, all of the bisimplicial structure maps
preserve arbitrary weak equivalences as well as all homotopy (co)limits, and restricting to 
strong left (or right) Quillen functors as maps still yields a bisimplicial model category. 
Notice this ability is technically vital here, since the stated goal is to recover
a bisimplicial set by restricting to objects in $\leftmodup[n,m]$ with entries in $\catu$ that are
in addition cofibrant.  There would be no reason to expect cofibrancy to be preserved
by all of these face and degeneracy maps
without something akin to the fact that the structure maps are all strong left Quillen functors.

First, the required construction of `extra degeneracy functors' for a categorical nerve.

\begin{defn}
	\label{extradegens}
	Suppose $\catm$ is a category with both initial ($\emptyset$) and final 
	($*$) objects.  Then in the
	categorical nerve with $\mathbb N_n(\catm) = Fun([n],\catm)$ there are two 
	additional functors $\overline s_{-1}, \overline s_{n}:N_{n-1}(\catm) \to N_n(\catm)$
	given by 
\begin{align*}
&\overline s_{-1}(X_0 \to X_1 \to \dots \to X_{n-1}) = 
	\emptyset \to X_0 \to X_1 \to \dots \to X_{n-1} \text{ and } \\
&\overline s_{n}(X_0 \to X_1 \to \dots \to X_{n-1}) = 
	X_0 \to X_1 \to \dots \to X_{n-1} \to {*}
\end{align*}	
\end{defn}
 
 Now, the following purely categorical observation, that does not seem to be well-known, 
 establishes the requisite underlying adjunctions.
 
 \begin{lemma}
 	\label{simpadjt}
 	If $\catm$ is a category with both initial and final objects, then 
	the simplicial category $\mathbb N_n(\catm)$ (or categorical nerve)
	has the property that for $0 \leq i \leq n-1$ both $(d_i,s_i)$ and $(s_i,d_{i+1})$,
	in addition to $(\overline s_{-1},d_0)$ and $(d_n,\overline s_n)$, form adjoint
	pairs as indicated.	
	Furthermore, the only new entries introduced by any of these functors
	are the initial and final objects.
 \end{lemma}
 
 \begin{proof}
 	First recall that $d_i$ is defined by removing the object $i$, through composition if
	$i$ is neither $0$ nor $n$.  In the same manner, $s_i$ for $0 \leq i \leq n-1$
	is defined by inserting the identity on the $i$-th object. 
	In all cases, it is straightforward to verify the adjoint property directly.
	For example, to see that $(d_i,s_i)$ and
	$(s_i,d_{i+1})$ for $0 \leq i \leq n-1$	are adjoint pairs between 
	$\mathbb N_{n}(\catm)$ and $\mathbb N_{n-1}(\catm)$, observe that
	the existence of a commutative diagram of the form
\diagramit{
X_0 \ar[d] \ar[r] & \dots \ar[r] & X_{i-1} \ar[d] \ar[r] & X_i \ar[d] \ar[r] & X_{i+1} \ar[d] \ar[r] 
	& X_{i+2} \ar[d] \ar[d] \ar[r] & \dots \ar[r] &X_{n}\ar[d] 	\\
Y_0 \ar[d] \ar[r] &  \dots \ar[r] &Y_{i-1} \ar[d] \ar[r] & Y_i \ar[d] \ar[r]^{=} & Y_i \ar[d] \ar[r] 
	& Y_{i+1} \ar[d] \ar[r] & \dots \ar[r] & Y_{n-1}\ar[d] \\
Z_0 \ar[r] &  \dots \ar[r] & Z_{i-1} \ar[r] & Z_i \ar[r] & Z_{i+1} \ar[r] & Z_{i+2}  \ar[r] & \dots \ar[r] & Z_n
}
	where the second row corresponds to $s_i(Y)$, is equivalent to the existence
	of a commutative subdiagram
\diagramit{
X_0 \ar[d] \ar[r] & \dots \ar[r] & X_{i-1} \ar[d] \ar[r] & X_{i+1} \ar[d] \ar[r] & X_{i+2} \ar[d] \ar[r] 
	& \dots \ar[r] & X_{n}\ar[d] \\
Y_0 \ar[d] \ar[r] &  \dots \ar[r] &Y_{i-1} \ar[d] \ar[r] & Y_i \ar[d] \ar[r] & Y_{i+1} \ar[d] \ar[r] 
	& \dots \ar[r] & Y_{n-1}\ar[d] \\
Z_0 \ar[r] &  \dots \ar[r] & Z_{i-1} \ar[r] & Z_i \ar[r] & Z_{i+2} \ar[r] & \dots \ar[r] & Z_n
}
	where the top row corresponds to $d_i(X)$ and the bottom row corresponds to 
	$d_{i+1}(Z)$.

	To see that $(d_n,\overline s_n)$ forms an adjoint pair, observe
	the existence of a commutative diagram of the form
\diagramit{
Y_0 \ar[d] \ar[r] & Y_1 \ar[d] \ar[r] & \dots \ar[r] &Y_{n-2} \ar[d] \ar[r]
	&Y_{n-1} \ar[d] \ar[r] & Y_n \ar[d] \\
Z_0 \ar[r] & Z_1 \ar[r] & \dots \ar[r] & Z_{n-2} \ar[r] & Z_{n-1} \ar[r] & {*}
}
	where the bottom row corresponds to 
	$\overline s_n(Z)$, is equivalent to the existence
	of a commutative subdiagram
\diagramit{
Y_0 \ar[d] \ar[r] & Y_1 \ar[d] \ar[r] & \dots \ar[r] &Y_{n-2} \ar[d] \ar[r] & Y_{n-1} \ar[d] \\
Z_0 \ar[r] & Z_1 \ar[r] & \dots \ar[r] & Z_{n-2} \ar[r] & Z_{n-1}
}
	where the top row corresponds to $d_n(Y)$.	
	
	Similarly, to see that $(\overline s_{-1},d_0)$ forms an adjoint pair, observe
	the existence of a commutative diagram of the form
\diagramit{
{\emptyset} \ar[d] \ar[r] & Y_0 \ar[d] \ar[r] & Y_1 \ar[d] \ar[r] & \dots \ar[r] &Y_{n-1} \ar[d] \\
Z_0 \ar[r] & Z_1 \ar[r] & Z_2 \ar[r] & \dots \ar[r] & Z_n
}
	where the top row corresponds to 
	$\overline s_{-1}(Y)$, is equivalent to the existence
	of a commutative subdiagram
\diagramit{
Y_0 \ar[d] \ar[r] & Y_1 \ar[d] \ar[r] & \dots \ar[r] & Y_{n-1} \ar[d] \\
Z_1 \ar[r] & Z_2 \ar[r] & \dots \ar[r] & Z_{n}
}
	where the bottom row corresponds to $d_0(Z)$.	
\end{proof}
 
 Given a model category of diagrams $\catm^\catc$ and a small
 subcategory $\catu \subset \catm$, let
 $ev_{\catu}$ denote the set of diagrams whose entries all lie 
 in $\catu$, while $ev_{\catu}^{cof}$ is the subset of such diagrams which are also cofibrant
 in $\catm^\catc$. 
  
\begin{thm}
	\label{bisimp}
	For any pointed model category $\catm$, there is a bisimplicial (pointed) model 
	category $\leftmodup[n,m]$ where 
	for each $n$ (and choice of horizontal or vertical) the structure maps for
	$0 \leq i \leq n-1$ both $(d_i,s_i)$ and $(s_i,d_{i+1})$,
	in addition to $(\overline s_{-1},d_0)$ and $(d_n,\overline s_n)$, form strong
	Quillen pairs as indicated.		
	Furthermore, for any (small) complete Waldhausen subcategory $\catu$ inside $\catm$,
	applying $ev_{\catu}^{cof}$ everywhere yields a bisimplicial set which is a model
	for the algebraic K-theory of $\catu$.
\end{thm}

\begin{proof}
	First, observe that from the Cartesian closed property for $\cat$, and the fact that
	the indexing category $\catc_{n,m} \approx [n] \times [m]$, Lemma \ref{simpadjt}
	implies each face and degeneracy map of this bisimplicial category 
	in either direction is part of an adjoint pair as indicated in the statement.
	
	To show these are all strong Quillen pairs, first
	consider the case $(\overline s_{-1},d_0)$.
	In this case, the latching maps at previously existing objects remain unchanged, since
	for any object $\alpha$ formerly in $\catc_0$ (the first row or column)
	one has $\alofx{X}$ the initial object.  
	As all latching maps at the newly added $\catc_0$ objects are identities, it follows that
	$\overline s_{-1}$ preserves (acyclic) cofibrations.  Having handled the exceptional
	case of $(\overline s_{-1},d_0)$, it will now suffice to show each 
	right adjoint other than $d_0$, but including $\overline s_n$,
	(vertical or horizontal) preserves (acyclic) fibrations.
	
	Notice fibrations are defined entrywise for these monotone increasing
	indexing categories.  Hence, both omitting and repeating entries, or inserting the
	identity on final objects, will preserve fibrations.
	Similarly, repeating entries, inserting the identity on final objects, 
	or omitting entries other than the first row or first column 
	will preserve weak equivalences.  Since $d_0$ (omitting the first row or column)
	is excluded at this point, the result is that each right adjoint currently under
	consideration is a strong right Quillen functor, so each left adjoint is also a strong 
	left Quillen functor. 
	
	The second statement now follows from Prop \ref{getT} below, keeping in mind that
	Lemma \ref{simpadjt} together with the first statement implies the simplicial structure 
	maps all commute with $ev_{\catu}^{cof}$, which thereby yields a bisimplicial set.
\end{proof}

Notice that each simplicial structure map above must preserve all weak 
	equivalences as both strong left and strong right Quillen functors, since one can factor an 
	arbitrary weak equivalence as an acyclic cofibration followed by an acyclic fibration. 
	It also follows that the derived functors remain adjoints at the level of the homotopy
	categories, so this structure is fairly rigid.	

\begin{remark}
	In fact, the construction of a bisimplicial set above works just as well with $\catu$
	the full subcategory on any set of cofibrant objects, without
	assuming it is a Waldhausen subcategory.  No claims are made here about the 
	properties of such an extension, but it could provide some flexibility in working
	with algebraic K-theory slightly outside of Waldhausen's original context.
\end{remark}

\subsection{Models for algebraic K-theory}

Now the topic shifts to recovering the relation between the construction above
and Thomason's variant of Waldhausen's construction.  The complications of dealing
with Waldhausen's construction directly are avoided in this way,
although they are outlined in Rem. \ref{stillsdot}.  
For a Waldhausen subcategory $\catu$ of $\catm$, 
let $T_n \catu$ denote the full subcategory of $\mathbb N_n(\catu)$ whose objects have 
each $X_i \to X_{i+1}$ a cofibration (in $\catm$), or equivalently, 
the (standard) Reedy cofibrant objects as defined in $\mathbb N_n(\catm)$ whose 
entries all lie in $\catu$.
Here a morphism
of $T_n \catu$ is called a cofibration if it is a (standard) Reedy cofibration (considered in
$\mathbb N_n(\catm)$.  Then Thomason's notion of weak equivalence $w$ is the class of
maps with $\latch[i]{f}$ a weak equivalence for each $i>0$, so his acyclic cofibrations
$\overline{w}$ in this case are
precisely the cofibrations from $\leftmod$ where $\catc=[n]$ and $\catc_0=[0]$.  Just to be clear, 
$\overline{w}$ consists of natural transformations between functors $[n] \to \catm$ where
the zero entry is a cofibration and all higher latching maps are acyclic cofibrations in $\catm$.

One now needs to observe that Waldhausen's proof (using Quillen's Theorem A)
applies in this case as well to see
$\overline{w}$ is `big enough' to lead to the algebraic K-theory space.  

\begin{lemma}
	\label{wbart}
	There is a homotopy equivalence between the bisimplicial sets 
	$N_*\overline{w}T_\bullet \catu$ and $N_*{w}T_\bullet \catu$.  Thus,
	the former also yields a model for the algebraic K-theory space of $\catu$.
\end{lemma}

\begin{proof}
	The first statement follows the same proof as for Lemma 1.6.3 of \cite{Wald}, 
	using Quillen's Theorem A.  The second
	then follows from the end of section 1.3 of \cite{Wald}, where $wT_n \catu$ is discussed
	(although $T_n \catu$ is never made explicit).
\end{proof}

 Now one can show the model category approach yields an enrichment of Thomason's 
 approach.
 
\begin{prop}
	\label{getT}
	There is a homotopy equivalence between the bisimplicial sets
	$ev_{\catu}^{cof} \leftmodup[*,\bullet]$ and $N_*{w}T_\bullet \catu$, so the former
	yields a model for the algebraic K-theory space of $\catu$.
\end{prop}

\begin{proof}
	In light of Lemma \ref{wbart}, it suffices to show $ev_{\catu}^{cof} \leftmodup[n,m]$ is 
	isomorphic, as a bisimplicial set, to $N_*\overline{w}T_\bullet \catu$.  Since in all
	cases the bisimplicial structure maps are inherited from 
	$\mathbb N_* \mathbb N_\bullet \catm$, which contains both, it will suffice to
	observe that for a fixed $(n,m)$ the subsets coincide.
	
	Consider the diagrams of the following form (all squares indicated,
	even if distorted, are pushouts used to define the latching objects)
\diagramit{
X_{i0} \ar[r] \ar[d] & X_{i1} \ar[d] \ar[r] & X_{i2} \ar[d] \ar[r] & \dots  \ar[r] & X_{im} \ar[d] \\
X_{(i+1)0} \ar[r] & \mathcal L_{(i+1)1} \ar[d]^{\sim} & \mathcal L_{(i+1)2} \ar[d]^{\sim}
	& \dots   & \mathcal L_{(i+1)m} \ar[d]^{\sim} \\
& X_{(i+1)1} \ar[ur] & X_{(i+1)2} \ar[ur] & \dots \ar[ur] & X_{(i+1)m}
}
	where objects are all in $\catu$, all maps are cofibrations, and those labeled with 
	$\sim$ are also weak equivalences.  By definition of $\overline{w}$, this set serves 
	for $0 \leq i<n$ as the vertical map $i \to i+1$ of an entry of $N_n \overline{w}T_m \catu$.
	However, keeping in mind that for
	cofibrant objects in $\leftmodup[n,m]$
	the latching maps are all cofibrations and those whose target is outside 
	the first row and first column are acyclic cofibrations, this same form
	of diagram serves as any row of a
	cofibrant object of $\leftmodup[n,m]$ with entries in $\catu$.
	Hence by induction up to $n$, the sets in question coincide.
\end{proof}

\begin{remark}
	\label{stillsdot}
It is also possible to approximate Waldhausen's $S_\bullet$ construction directly, by
working with $\catc_0$ the top row and shifting $\leftmodup[n,m]$ to the entry $(n+1,m)$.
Of course, this runs into the usual difficulty of how to define $d_0^h$ for $S_\bullet$.  In fact,
in line with Lemma \ref{simpadjt}, one can view $d_0^h$ as simply a left adjoint to $s_0^h$ (with
some additional technical conditions, see \cite{myK}).  
As a consequence, Waldhausen's introduction of quotients
becomes quite natural, as one can see from the following smaller example in a categorical
nerve.  The key observation is that a commutative diagram of the form
\diagramit{
X_0 \ar[d] \ar[r] & X_1 \ar[d] \ar[r] & X_2 \ar[d] \\
{*} \ar[r] & Y_0 \ar[r] & Y_1
}
where the second row represents Waldhausen's $s_0 Y$, 
is equivalent to a commutative diagram
\diagramit{
X_1/X_0 \ar[d] \ar[r] & X_2/X_0 \ar[d] \\
Y_0 \ar[r] & Y_1 
}
so one chooses the top row of this second diagram as $d_0(X)$ to get the left adjoint property.
To justify this statement notice that the first commutative square in the top diagram is 
equivalent to the existence of the first vertical map in the bottom diagram.  
Then the large commutative 
rectangle in the top diagram is equivalent to the existence of the second vertical map in
the bottom diagram.  Finally, the commutativity of the second square in the top diagram
is now equivalent to commutativity of the bottom diagram, since both vertical maps in the
second square factor through the quotients and the top square in
\diagramit{
X_1 \ar[d] \ar[r] & X_2 \ar[d] \\
X_1/X_0 \ar[d] \ar[r] & X_2/X_0 \ar[d] \\
Y_0 \ar[r] & Y_1
}
commutes (is even a pushout) by the fact that $X_0 \to X_2$ factors through $X_0 \to X_1$. 

Of course, a left adjoint is only unique (some might say only defined)
up to natural isomorphism, so this would lead to something
not quite a simplicial category, where $d_0$ is only well-defined up to `homotopy'.
However, Waldhausen's requirement of keeping track of choices of quotients for $S_\bullet$
may be viewed as
a special case of an explicit rectification functor, which minimally modifies this construction
to produce an actual simplicial object in small categories or even in Waldhausen categories
with exact functors as morphisms.  
This small rectification construction is described concretely using descending sequences 
of objects together with certain choices of isomorphisms, in an inductive manner.  Of course,
one could instead appeal to modern technology for rectifying all kinds of pseudo-diagrams, 
but that would lose both the explicit nature of the construction and the historical context.  
See \cite{myK} for complete details. 
\end{remark}

\section{Examples of Applications to Localizations}

This brief, informal
section is mainly intended to discuss two very simple indexing categories and the
many model structures one can produce by the techniques of this article.

\subsection{Arrow categories}

Begin by considering $\catm(\to)$, so $\catc=\{0 \to 1\}=[1]$, for an arbitrary model category
$\catm$.  Notice one has the standard Reedy structure (with $\catc_0=\catc$), 
and $\catc$ is monotone increasing
so any full subcategory is left acceptable.  Thus, one also has $\leftmod$ with $\catc_0$ either 
${0}$ or ${1}$.  Finally, $\fnc[0]$ is always acceptable, so one has $\rightmod$ with 
$\catc_0={0}$ as well.  The last is the only one of these which is unusual in most cases,
while both $\leftmod$ and $\rightmod$ for $\catc_0={0}$ are Quillen equivalent to the original
$\catm$ by Rem. \ref{zero}.  If $\catm$ happens to be cofibrantly generated, one might 
be tempted to add to this list the 
modified projective structures, but they are already here by Rem. \ref{sameincrease}. 

Changing the choice of $\catc_0$ has not been considered until now, but the following
observation makes it tractable for certain comparisons.

\begin{lemma}
	\label{changesub}
	If $\catc_1 \subset \catc_0$, then the identity forms a strong right Quillen functor
	$\leftmod \to \leftmodc{1}$ (when both exist).  Dually, the identity forms a strong left Quillen
	functor $\rightmod \to \rightmodc{1}$ (when both exist).
\end{lemma}

\begin{proof}
	Since $\leftmod$ and $\leftmodc{1}$ both use the Reedy notion of fibration, it suffices
	to observe that weak equivalences in the first are always weak equivalences in the second
	by definition.  The other case is dual.
\end{proof}

Thus, one has four distinct model structures on $\catm(\to)$, and there are five 
more if one assumes
there is also a localization $\catm_f$ around.  As above, one would get four more model structures
by only considering $\catm_f$ and the techniques of this article.  However, by combining
the two structures, there is also one more.  Let $\catm_0=\catm$ and 
$\catm_1=\catm_f$.  Then $\leftmodc{}$ (notice $\catc_0=\catc$ as for the 
standard Reedy structure) gives something here referred to as a
`mixed structure' which is a convenient place to
study the localization map.  For example, if $X \in \catm$ is fibrant, then the localization
map $X \to L_f X$ becomes a fibrant replacment in the mixed structure
for the identity $1:X \to X$.  The mixed structure also relates nicely with looking at 
$\leftmod$ for $\catm$ with $\catc_0={0}$ and
for $\catm_f$ with $\catc_0={1}$ by Lemma \ref{changesub} and Prop. \ref{restrict}.

In fact, the mixed structure is the only reasonable way to combine the two structures 
using the techniques of this article, since it coincides with the
associated $\projmod$ by $\catc$ monotone increasing, any choice of $\catc_0 \neq \catc$
is covered by the eight cases related to a single structure on $\catm$, 
and the reversed choice ($\catm_0=\catm_f$ and $\catm_1=\catm$) does not
seem to satisfy the compatibility condition.

\subsection{Commutative Squares}

One can also proceed as above for $\catc$ a commutative square 
\diagramit{
00 \ar[r] \ar[d] & 01 \ar[d] \\
10 \ar[r] & 11
}
as well, which is again clearly monotone increasing.  
However, a new variant is now available, if one has
two different localizations of $\catm$, say $\catm_f$ and $\catm_g$, together with the
`common localization' $\catm_h$.  It will not be necessary to be precise about what $\catm_h$ 
should mean or when it exists.  However, it should represent localization with respect to both
$f$ and $g$ at the same time, so the idea is that any cofibration which is acyclic in
either $\catm_f$ or $\catm_g$ should now be acyclic in $\catm_h$.  Now take 
$\catc_0=\catc$ (again like the standard Reedy structure) 
with $\catm_{00}=\catm$, $\catm_{01}=\catm_f$, $\catm_{10}=\catm_g$
and $\catm_{11}=\catm_h$ to form a `highly mixed structure'.  
In this case, one has a `localization square'
\diagramit{
X \ar[d] \ar[r] & L_f X \ar[d] \\
L_g X \ar[r] & L_h X
}
which arises as a fibrant replacement in the highly mixed structure
for the constant square on a fibrant object $X \in \catm$.  Perhaps the usual localization 
square for topological spaces, with the far corner the rationalization and one near corner
a $p$-localization could lead to some valuable insights via this highly mixed structure.
By Lemma \ref{changesub}, it also relates well with the two mixed structures associated 
to $\leftmod$ by picking $\catc_0=\{00 \to 01 \}$ or $\{ 00 \to 10 \}$ and these same choices
for $\catm_{ij}$.  Among other things, these are model structures on the category of 
squares which are closely related to just one of the two original localizations by 
Prop. \ref{restrict}.

It seems likely that further examples along these lines could be used to understand 
successive localizations in a highly structured form.  Perhaps the relationship between
the highly mixed structure for two localizations and the mixed structure for just one of them
could lead to an inductive framework for studying successive localizations.  One interesting
source of examples could be the smashing localizations of the stable homotopy category,
where localization squares already appear prominently.

\appendix
\section{The Reedy Inductive Framework}
\label{inductsec}

The goal of this appendix is to provide a self-contained introduction to the 
inductive framework available for diagrams indexed by a Reedy category, 
leading up to the proof of Prop. \ref{liftbylatch}.  This is provided since
 detailed proofs of these statements have often been omitted
in the literature, and the current claim is just a bit beyond the scope of previous claims.  
No model category machinery beyond the basic concepts of 
factoring a map and lifting in a commutative square are considered here, so these arguments
apply equally well to arbitrary weak factorization systems. 

The inductive framework is embodied by the next five technical lemmas,
which precede the proof of the proposition.

\begin{lemma}
	\label{inductobj}
	A diagram $X:\fnc \to \catm$ is equivalent to a restricted diagram $\widehat X:\fnc[n-1] \to \catm$
	together with a choice for each $\alpha \in \catc$ of degree $n$ of a factorization
\diagramit{
\alofx{X} \ar[r] \ar[dr] & X_\alpha \ar[d] \\
&  \amofx{X}
}
	of the canonical map.
\end{lemma}

\begin{proof}
	The existence of $X$ clearly implies that of $\widehat X$ and the relevant factorizations
	are provided by the canonical morphisms.  Now suppose $\widehat X$ 
	and such factorizations
	have been chosen.  Since the value of $X$ on objects is thereby specified, it suffices to see
	how to define $X$ on all morphisms of $\fnc$, in a way which preserves compositions
	(and identities).  Suppose
	$\psi:\alpha \to \beta$ is a morphism of $\fnc$ which is not an identity.  
	By the Reedy condition,
	there exists a unique factorization $\alpha \to \delta \to \beta$ with $\alpha \to \delta$
	in $\catc^-$ and $\delta \to \beta$ in $\catc^+$, so
	define $X(\psi)$ by the composite
\[X_\alpha \to \amofx{X} \to X_\delta \to \alofx[\beta]{X} \to X_\beta  .
\]
	Here the map $\amofx{X} \to X_\delta$ is that of the limiting system defining the source, 
	and similarly for the target, so the diagram
\diagramit[\label{triangles}]{
X_\alpha \ar[dr] \ar[r] & \amofx{X} \ar[d] & \text{or} & \alofx[\beta]{X} \ar[r] & X_\beta\\
 & X_\delta && X_\delta \ar[ur] \ar[u]
}	
	commutes whenever $|\alpha| < n$ or $|\beta|<n$
	by the assumption that $\widehat X$ is an $\fnc[n-1]$-indexed diagram.
	
	In order to see the choice above is compatible with composition, 
	there are several cases and 
	full detail will only be given for the most complex, as the others are simpler variations
	on this argument.  Suppose 
	$\varphi: \beta \to \gamma$ which factors as $\beta \to \epsilon \to \gamma$,
	in addition to $\psi$ as above, with $|\alpha|=|\beta|=|\gamma|=n$.  
	Then composition closure of
	the monotone maps and the unique factorization
	in a Reedy structure implies the unique factorization of 
	$\delta \to \beta \to \epsilon$ as $\delta \to \sigma \to \epsilon$ leads to a commutative
	diagram in $\fnc$
\diagramit{
\alpha \ar[dr] && \beta \ar[dr] && \gamma \\
& \delta \ar[ur] \ar[dr] && \epsilon \ar[ur] \\
&& \sigma \ar[ur]  
}
	with $\alpha \to \sigma \to \gamma$ the unique decomposition of the composite.
	To see $X(\varphi \circ \psi)=X(\varphi)\circ X(\psi)$ consider the following commutative
	diagram
\diagramit{
X_\alpha \ar[r] & \amofx{X} \ar[dr] & \alofx[\beta]{X} \ar[r] & X_\beta \ar[r] &\amofx[\beta]{X} \ar[d] 
& \alofx[\gamma]{X} \ar[r] & X_\gamma \\
&& X_\delta \ar[u] \ar[dr] && X_\epsilon \ar[ur] \\
&&& X_\sigma \ar[ur]  
}	
	where the `low path' represents $X(\varphi \circ \psi)$ by the decomposition 
	statement above, while the `high path' represents $X(\varphi)\circ X(\psi)$.
	Commutativity of the middle pentagon above follows from the assumption 
	of a choice of such a factorization for each $\beta$ and the fact that
	the canonical map $\alofx[\beta]{X} \to \amofx[\beta]{X}$ will factor $X_\delta \to X_\epsilon$,
	since it is defined essentially as the combination of all such maps.	  
	
	In the other cases, one or more pieces of this large diagram can be simplified by 
	applying the commutative triangles indicated in \eqref{triangles}.  
	As one other example, the diagram when $|\alpha| <n$ and
	$|\beta|<n$ is the following.
\diagramit{
X_\alpha \ar[rr] \ar[dr] && X_\beta \ar[dr] && \alofx[\gamma]{X} \ar[r] & X_\gamma \\
& X_\delta \ar[ur] \ar[dr] && X_\epsilon \ar[ur] \\
&& X_\sigma \ar[ur]  
}
\end{proof}

\begin{lemma}
	\label{inductmap}
	Suppose $X$, $Y:\fnc \to \catm$.  A morphism $f:X \to Y$ in $\fnc$
	is equivalent to a restricted morphism
	$\widehat f:\widehat X \to \widehat Y$ of diagrams indexed on $\fnc[n-1]$
	together with a choice for each $\alpha \in \catc$ of degree $n$ of (the vertical arrows in)
	a commutative square
\diagramit[\label{mapsquare}]{
X_\alpha \ar[r]^{\blofx{f}} \ar[d]_{\match{f}} &  \lofx{f} \ar[d]^{\latch{f}} \\
\mofx{f} \ar[r]_{\bmofx{f}} & Y_\alpha   .
}
\end{lemma}

\begin{proof}
	One implication is clear as indicated by the labels on the given maps, 
	so assume the existence of $\widehat f$ and the indicated choices
	of the vertical arrows in the commutative squares as labeled.  Then define $f_\alpha$
	as either composite $X_\alpha \to Y_\alpha$, which agrees with the definition of 
	$\widehat f$
	by construction whenever $|\alpha|<n$.  In order to verify the transformation so
	defined is natural, suppose $\psi:\alpha \to \beta$ is a morphism in $\fnc$ with 
	$max\{|\alpha|,|\beta| \}=n$ (as otherwise $\psi \in \fnc[n-1]$). 	
	Begin with the case $|\alpha|=n>|\beta|$ and consider the diagram
\diagramit{
X_\alpha \ar[dd] \ar[rr]^{f_\alpha} \ar[dr]^{\match{f}} && Y_\alpha \ar[dd] \\
& \mofx{X} \ar[ur]^{\bmofx{f}} \ar[dl] \\
\amofx{X} \ar[d] \ar[rr]^{\amofx{f}} && \amofx{Y} \ar[d] \\
X_\beta \ar[rr]_{f_\beta} && Y_\beta
} 
	where the top square commutes since $\mofx{X}$ is the pullback of the 
	distorted square, while the bottom square commutes by naturality of the limit 
	defining $\amofx{X}$.  Thus, the outer rectangle shows this case of naturality
	of $f$ with respect to $\psi$.  The case where $|\alpha|<n=|\beta|$ is dual.   
	If both $\alpha$ and $\beta$ have degree $n$, then the unique
	factorization providing an intermediate object with lower degree
	(and the ability to compose commutative squares) reduces to the combination of the 
	two cases considered already.  
\end{proof}

\begin{defn}
	Suppose $g:X \to Y$ is a morphism of diagrams $\fnc \to \catm$ and 
	there exists a restricted factorization as $\fnc[n-1]$-indexed diagrams 
	$\widehat X \stackrel{\widehat f}{\to} \widehat Z \stackrel{\widehat p}{\to} \widehat Y$
	(so $\widehat Z$ is only a $\fnc[n-1]$-indexed diagram).  Then there is a
	canonical map $\lofx{f} \to \mofx{p}$ induced by the 
	universal properties of the pushout and pullback from the diagram
\diagramit{
\alofx{X} \ar[d] \ar[r]^{\alofx{f}} & \alofx{Z} \ar[d] \ar[r]^{\alofx{p}} & \alofx{Y} \ar[d] \\
X_\alpha \ar[r]^{\blofx{f}} \ar[d]_{\match{g}} & \lofx{f} \ar@{.>}[d] 
	\ar[r] & \lofx{g} \ar[d]^{\latch{g}} \\
\mofx{g} \ar[r] \ar[d] & \mofx{p} \ar[d] \ar[r]^{\bmofx{p}} & Y_\alpha \ar[d] \\
\amofx{X} \ar[r]^{\amofx{f}} & \amofx{Z} \ar[r]^{\amofx{p}} & \amofx{Y} 
}
	where the middle rectangle commutes by Lemma \ref{inductmap} and the fact
	that $g$ extends.
\end{defn}

\begin{lemma}
	\label{inductfactor}
	Suppose $g:X \to Y$ is a morphism of diagrams $\fnc \to \catm$.  A factorization
	of $g$ as $X \stackrel{f}{\to} Z \stackrel{p}{\to} Y$ in $\fnc$
	is equivalent to a restricted factorization as $\fnc[n-1]$-indexed diagrams 
	$\widehat X \stackrel{\widehat f}{\to} \widehat Z \stackrel{\widehat p}{\to} \widehat Y$
	(so $\widehat Z$ is only a $\fnc[n-1]$-indexed diagram)
	together with a choice, for each $\alpha$ in $\catc$
	of degree $n$, of a factorization of the canonical map
\diagramit{
\lofx{f} \ar[r]^-{\latch{f}} & Z_\alpha \ar[r]^-{\match{p}} & \mofx{p}   .
}
\end{lemma}	

\begin{proof}
	Once again, one direction is straightforward, so assume the existence of the restricted
	factorization and the indicated choices.  Apply Lemma \ref{inductobj} using the fact
	that the assumption here provides $Z_\alpha$ with a choice of factorization of
	the canonical map $\alofx{Z} \to \lofx{f} \to \mofx{p} \to \amofx{Z}$ to see 
	the intermediate object $Z$
	so constructed actually forms an $\fnc$-indexed diagram.  To see the indicated map
	$\widehat f$ extends to $\fnc$, it suffices by Lemma \ref{inductmap} to show the
	chosen map $\latch{f}$ induces a choice of $\match{f}$ making a square as in
	\eqref{mapsquare} commute.  
	As the corresponding statement and its justification for $\widehat p$ is strictly dual,
	only this case will be considered.
	
		To keep track of the entire argument, 
	consider the following commutative diagram
\diagramit[\label{bigrect}]{
\alofx{X} \ar[d] \ar[r]^{\alofx{f}} \ar@/^6ex/[rr]^{\alofx{g}}& \alofx{Z} \ar[d] \ar[r]^{\alofx{p}} 
	& \alofx{Y} \ar[d] \\
X_\alpha  \ar@{}[dr] |{\bf 4} \ar@{.>}[d]^{\match{f}} \ar[r]^{\blofx{f}} \ar@/_4ex/[dd]_{\match{g}} 
	& \lofx{f} \ar[d]^{\latch{f}} \ar[r] & \lofx{g} \ar[d] \ar@/^4ex/[dd]^{\latch{g}} \\
\mofx{f} \ar[r]^{\bmofx{f}} \ar[d]  \ar@{}[dr] |{\bf 3}& Z_\alpha \ar[d]^{\match{p}} \ar[r]^{\blofx{p}} 
	& \lofx{p} \ar@{.>}[d]_{\latch{p}} \\
\mofx{g} \ar[r] \ar[d]  \ar@{}[dr] |{\bf 2}& \mofx{p} \ar[d] \ar[r]^{\bmofx{p}} \ar@{}[dr] |{\bf 1} 
	& Y_\alpha \ar[d] \\
\amofx{X} \ar[r]^{\amofx{f}} \ar@/_6ex/[rr]^{\amofx{g}} & \amofx{Z} \ar[r]^{\amofx{p}} & \amofx{Y} 
}
	where the bottom squares labeled $1$ through $3$ will be shown to be pullbacks, 
	the top three squares are similarly pushouts, 
	and the two tall rectangles of the middle 2 by 2 square commute by 
	$g$ extending to $\fnc$ and Lemma \ref{inductmap}, along with the construction
	of the canonical map.
	
	Now consider the rectangle along the bottom, so $1$ is a pullback and the composite
	of $1$ and $2$ is a pullback, which implies $2$ is a pullback.  Once again, the
	composite of $2$ and $3$ is a pullback, with $2$ just shown to be a pullback, so
	$3$ is also a pullback, and its universal property will be used to define the 
	dotted arrow $\match{f}$.  The assumption
	of having a factorization of the canonical map as
	$\latch{f} \match{p}$ implies, by the definition of the canonical map,
	that the two longest paths in the composite of $3$ and $4$ commute, which suffices
	to induce a map making $4$ commute as in \eqref{mapsquare}.
\end{proof}

And finally, one has the induction step for constructing lifts in commutative squares, which
is the point of the next two lemmas.  The first of these two lemmas 
builds the relevant structure, which is analyzed in the second.

\begin{lemma}		
	\label{onlysquare}
	Suppose
\diagramit{
A \ar[d]_{f} \ar[r]^{g} & X \ar[d]^{p} \\
B \ar[r]_{h} \ar@{.>}[ur]^{\widehat k} & Y
}
is a (solid) commutative square of $\fnc$-indexed diagrams in $\catm$.  Then a dotted lift
$\widehat k$ in the restriction to $\fnc[n-1]$-indexed diagrams induces a commutative square
\diagramit{
\lofx{f} \ar[d]_{\latch{f}} \ar[r] & X_\alpha \ar[d]^{\match{p}} \\
B_\alpha \ar[r]  & \mofx{p}
}
for each $\alpha$ of degree $n$ in $\catc$.
\end{lemma}

\begin{proof}
	By exploiting universal properties of the pushout and pullback involved, it suffices
	to construct (for each $\alpha$ of degree $n$) a commutative diagram 
	(including the dotted maps) as follows
\diagramit{
\alofx{A} \ar[d]_{\alofx{f}} \ar[r] & A_\alpha \ar[d]_{f_\alpha} \ar[r]^{g_\alpha} &
	X_\alpha \ar[d]^{p_\alpha} \ar[r] & \amofx{X} \ar[d]^{\amofx{p}} \\
\alofx{B} \ar[r] \ar@/_2ex/@{.>}[urr] & B_\alpha \ar@/_2ex/@{.>} [urr] \ar[r]_{h_\alpha} & 
	Y_\alpha \ar[r] & \amofx{Y} .
}
	However, the assumption that $\widehat k$ is a restricted lift implies the diagram
\diagramit{
\alofx{A} \ar[d]_{\alofx{f}} \ar[r]^{\alofx{g}} & \alofx{X} \ar[d]^{\alofx{p}} \ar[r] & 
	X_\alpha \ar[d]^{p_\alpha} \\
\alofx{B} \ar[ur]^{\alofx{\widehat k}} \ar[r]_{\alofx{h}} & \alofx{Y} \ar[r] & Y_\alpha
}
	commutes, so it suffices to define the dotted map $\alofx{B} \to X_\alpha$ as the composite
	$\alofx{B} \stackrel{\alofx{\widehat k}}{\longrightarrow} \alofx{X} \to X_\alpha$.
	The construction of $B_\alpha \to \amofx{X}$ is dual.	
\end{proof}

Now one can see how to build lifts inductively by choosing lifts in the squares of
Lemma \ref{onlysquare}.

\begin{lemma}
	\label{inductlift}
	Suppose
\diagramit{
A \ar[d]_{f} \ar[r]^{g} & X \ar[d]^{p} \\
B \ar[r]_{h} \ar@{.>}[ur]^{k} & Y
}
is a (solid) commutative square of $\fnc$-indexed diagrams in $\catm$.  Then a dotted lift
$k$ of $\fnc$-indexed diagrams is equivalent to a restricted lift $\widehat k$ of 
$\fnc[n-1]$-indexed diagrams together with a choice of lift in the commutative square
\diagramit{
\lofx{f} \ar[d]_{\latch{f}} \ar[r] & X_\alpha \ar[d]^{\match{p}} \\
B_\alpha \ar[r] \ar@{.>}[ur]^{k_\alpha} & \mofx{p}
}
of Lemma \ref{onlysquare}
for each $\alpha$ of degree $n$ in $\catc$.  
\end{lemma}

\begin{proof}
	Once again, one direction is straightforward, so assume 
	the indicated choices of $k_\alpha$ are given and try to verify the extended lift exists.
	Since the indicated choices yield a lift at each entry by considering
\diagramit{
A_\alpha \ar[rr] \ar[d] && X_\alpha \ar[d]^{\match{p}} \\
\lofx{f} \ar[d]_{\latch{f}} \ar[urr] && \mofx{p} \ar[d] \\ 
B_\alpha \ar[rr] \ar[urr] \ar@{.>}[uurr]_{k_\alpha} && Y_\alpha  ,
}	
	it suffices to use Lemma \ref{inductfactor} to show $k f$ would yield a
	factorization of $g$ as $\fnc$-indexed diagrams, hence that $k$ is itself
	a morphism of $\fnc$-indexed diagrams.  Thus it remains to verify that
	the choice of $k_\alpha$ as indicated leads to a choice of factorization of the
	canonical map as follows
	$\lofx{f} \stackrel{\latch{f}}{\longrightarrow} B_\alpha 
	\stackrel{\match{k}}{\longrightarrow} \mofx{k}$, with the first map already defined. 
	
	However, 	notice the lift assumption implies 
\diagramit{
 & X_\alpha \ar[d]^{\match{p}} \ar[dr] \\
 B_\alpha \ar[ur]^{k_\alpha} \ar[r] & \mofx{p} \ar[r] & \amofx{X}
}
	commutes, which together with the construction of
	$B_\alpha \to \amofx{X}$ in the proof of Lemma \ref{onlysquare}, 
	implies the (solid) diagram
\diagramit{
\lofx{f} \ar[r]_{\latch{f}} & 
	B_\alpha \ar@{.>}[dr]_{\match{k}} \ar@/^2ex/[drr]^{k_\alpha} \ar@/_4ex/[ddr] \\
&& \mofx{k} \ar[d] \ar[r]_{\bmofx{k}} & X_\alpha \ar[d] \\
&& \amofx{B} \ar[r]_{\amofx{k}} & \amofx{X}
}
	also commutes.  This yields the indicated dotted arrow by the pullback property,
	and so the required factorization of the canonical map.
\end{proof}

The main point of this appendix has been to build up to the proof of Prop. \ref{liftbylatch} which
follows.  

\begin{proof}[Proof of Proposition \ref{liftbylatch}]
	Suppose 
\diagramit{
A \ar[d]_{f} \ar[r]^{g} & X \ar[d]^{p} \\
B \ar[r]_{h} \ar@{.>}[ur]^{k} & Y
}	
	a (solid) commutative square in $\catm^\catc$; proceed by induction to produce the 
	required lift $k$.  Recall that when $\alpha$ has degree zero, $\fnc[0]$ is a discrete 
	category, while $\latch{f} \approx f_\alpha$ and $\match{p} \approx p_\alpha$.  Thus,
	one can choose $k_0$, a lift in the restriction to $\fnc[0]$, through a choice of $k_\alpha$
	in each such (independent) square which is possible by assumption.  Now apply
	Lemma \ref{inductlift} to choose $k_n$ given $k_{n-1}$, by making
	choices of lifts for each $k_\alpha$ with $|\alpha|=n$, which is 
	again possible by assumption.
\end{proof}

\end{document}